\documentclass[11pt,oneside,reqno]{amsart}
\usepackage[utf8x]{inputenc}

\usepackage{textcomp}
\usepackage{amsmath,amssymb}
\usepackage{mathrsfs}
\usepackage[T2A]{fontenc}
\usepackage{amsthm}
\usepackage{epigraph}

\usepackage[linktocpage]{hyperref}
\usepackage{amsfonts,graphics,amsthm,amsfonts,amscd,latexsym}
\usepackage{epsfig}
\usepackage{flafter}
\usepackage{mathtools}
\usepackage{comment}
\usepackage{stmaryrd}
\usepackage{enumitem}
\usepackage{textcomp}
\usepackage{epigraph}
\usepackage{longtable}
\usepackage{tikz-cd}
\hypersetup{
colorlinks=true,
linkcolor=red,
citecolor=green,
filecolor=blue,
urlcolor=blue
}
\usepackage{tikz}
\usetikzlibrary{quotes,babel}

\usetikzlibrary{graphs,positioning,arrows,shapes.misc,decorations.pathmorphing}

\tikzset{
>=stealth,
every picture/.style={thick},
graphs/every graph/.style={empty nodes},
}

\tikzstyle{vertex}=[
draw,
circle,
fill=black,
inner sep=1pt,
minimum width=5pt,
]
\usepackage[position=top]{subfig}
\usepackage[alphabetic]{amsrefs}
\usepackage{amssymb}
\usepackage{color}

\calclayout

\usetikzlibrary{decorations.pathmorphing}
\tikzstyle{printersafe}=[decoration={snake,amplitude=0pt}]

\newcommand{\mult}{\operatorname{mult}}

\newcommand{\coreg}{\mathrm{coreg}\,}
\newcommand{\coregG}{\mathrm{coreg}_G\,}

\newcommand{\ZZ}{\mathbb{Z}}
\newcommand{\CC}{\mathbb{C}}
\newcommand{\QQ}{\mathbb{Q}}

\newcommand{\CG}{\ensuremath{\mathfrak{C}}}
\newcommand{\DG}{\ensuremath{\mathfrak{D}}}
\newcommand{\AG}{\ensuremath{\mathfrak{A}}}
\newcommand{\SG}{\ensuremath{\mathfrak{S}}}

\newcounter{theoremsec}[section]
\newcounter{theoremint}

\def\O#1.{\mathcal {O}_{#1}}
\def\pr #1.{\mathbb P^{#1}}
\def\af #1.{\mathbb A^{#1}}
\def\ses#1.#2.#3.{0\to #1\to #2\to #3 \to 0}
\def\xrar#1.{\xrightarrow{#1}}
\def\K#1.{K_{#1}}
\def\bA#1.{\mathbf{A}_{#1}}
\def\bM#1.{\mathbf{M}_{#1}}
\def\bL#1.{\mathbf{L}_{#1}}
\def\bB#1.{\mathbf{B}_{#1}}
\def\bK#1.{\mathbf{K}_{#1}}
\def\subs#1.{_{#1}}
\def\sups#1.{^{#1}}

\newcommand{\OOO}{\mathscr{O}}

\usepackage{tikz}
\usetikzlibrary{matrix,arrows,decorations.pathmorphing}

\newtheorem{theorem}{Theorem}[subsection]

\newtheorem{theoremintr}[theoremint]{Theorem}
\newtheorem{lemma}[theorem]{Lemma}
\newtheorem{lemmasec}[theoremsec]{Lemma}
\newtheorem{proposition}[theorem]{Proposition}
\newtheorem{propositionsec}[theoremsec]{Proposition}
\newtheorem{propositionint}[theoremint]{Proposition}
\newtheorem{corollary}[theorem]{Corollary}
\newtheorem{corollarysec}[theoremsec]{Corollary}
\newtheorem{corollaryint}[theoremint]{Corollary}

\newtheorem{conjecture}[theorem]{Conjecture}

\theoremstyle{definition}
\newtheorem{definition}[theorem]{Definition}
\newtheorem{definitionsec}[theoremsec]{Definition}
\newtheorem{example}[theorem]{Example}
\newtheorem{examplesec}[theoremsec]{Example}

\newtheorem{questionsec}[theoremsec]{Question}

\newtheorem{remark}[theorem]{Remark}
\newtheorem{remarksec}[theoremsec]{Remark}

\theoremstyle{remark}

\numberwithin{equation}{section}
\usepackage[all]{xy}
\newcounter{rownumber}[figure]
\newcounter{rownumber-irr}[figure]
\newcounter{rownumber-p1}[figure]
\begin{document}

\title{$G$-coregularity of del Pezzo surfaces}

\author[K.~Loginov]{Konstantin Loginov}

\author[V.~Przyjalkowski]{Victor Przyjalkowski}

\author[A.~Trepalin]{Andrey Trepalin}

\address{\emph{Konstantin Loginov}
\newline
\textnormal{Steklov Mathematical Institute of Russian Academy of Sciences, Moscow, Russia.}
\newline
\textnormal{National Research University Higher School of Economics, Russian Federation, Laboratory of Algebraic Geometry, NRU HSE.}
\newline
\textnormal{Laboratory of AGHA, Moscow Institute of Physics and Technology.}
\newline
\textnormal{\texttt{loginov@mi-ras.ru}}}

\address{\emph{Victor Przyjalkowski}
\newline
\textnormal{Steklov Mathematical Institute of Russian Academy of Sciences, Moscow, Russia.}
\newline
\textnormal{National Research University Higher School of Economics, Russian Federation, Laboratory of Mirror Symmetry, NRU HSE.}
\newline
\textnormal{\texttt{victorprz@mi-ras.ru, victorprz@gmail.com}}}
\thanks{
        \newline
		}
		
\address{\emph{Andrey Trepalin}
\newline
\textnormal{Steklov Mathematical Institute of Russian Academy of Sciences, Moscow, Russia.}
\newline
\textnormal{National Research University Higher School of Economics, Russian Federation, Laboratory of Algebraic Geometry, NRU HSE.}
\newline
\textnormal{\texttt{trepalin@mccme.ru}}}
		
\maketitle

\epigraph{\quad \quad \quad \quad \quad \quad \quad \quad \quad \quad \quad \quad \quad \quad \quad \quad \quad \quad \quad \quad \quad \quad \quad \quad \emph{Это ж-ж-ж --- неспроста!}\\ \quad \quad \quad \quad \quad \quad \quad \quad \quad \quad \quad \quad \quad \quad \quad \quad \quad \quad \quad \quad \quad \quad \quad \quad \, Винни-Пух}

\begin{abstract}
We introduce and study the notion of $G$-coregularity of algebraic varieties endowed with an action of a finite group $G$.
We compute $G$-coregularity of smooth del Pezzo surfaces of~degree at~least~$6$, and give a characterization of groups that can act on conic bundles \mbox{with $G$-coregularity~$0$}.
We describe the relations between the notions of $G$-coregularity, $G$-log-canonical thresholds, $G$-birational rigidity, and exceptional quotient singularities.
\end{abstract}
\setcounter{tocdepth}{1}

\tableofcontents
\section*{Introduction}
We work over an algebraically closed field $\mathbb{K}$ of characteristic $0$.

A smooth (or mildly singular) projective variety $X$ is called a Fano variety if its anti-canonical class $-K_X$ is ample.
%According to the minimal model program, Fano varieties can be considered as one of the building blocks of smooth projective varieties.
A natural way to study its geometry is by looking at the elements of the anti-pluricanonical linear system $|-mK_X|$ for $m\geq 1$. The classical question here is: how singular can these elements be?

We answer this question  in terms of coregularity. The definition of coregularity is given using the~following invariants. The dual complex of a divisor on a variety is a topological space that captures its combinatorial complexity, see Section \ref{sec-dual-complex} for the definition.
The notion of regularity of~a~log-canonical pair $(X, D)$ on a variety $X$ was introduced by V. Shokurov in \cite[Proposition-Definition  7.11]{Sh00}. By~definition, it is the dimension of the dual complex of the boundary divisor $D$ of this pair, computed on~some log resolution.
V. Shokurov also introduced the notion of complete regularity of a pair $(X, D)$, which means the~maximum of regularities of all pairs that consist of log-canonical complements of $K_X+D$.
For brevity, if $X$ is a normal projective variety with at worst klt singularities,
we use the term regularity $\mathrm{reg}(X)$ of $X$ instead of complete regularity of $(X, 0)$, see Definition~\ref{defin-regularity} for details. The coregularity, defined by J. Moraga  \cite{Mo24}, is the dual notion defined as
\[
\mathrm{coreg}(X)=\dim X - 1 - \mathrm{reg}(X).
\]
The coregularity of a variety $X$ belongs to the set $\{0,1,\ldots, \dim X\}$ (here we use the con\-ven\-tion~\mbox{$\dim \varnothing = -1$}). In particular, in the case of coregularity $0$ the dual complex of some divisor which is $\mathbb{Q}$-linearly equivalent to the anti-canonical divisor on $X$ has maximal possible dimension. From several points of view, this case seems to be most interesting.
%The opposite case corresponds to exceptional Fano varieties.
It is expected that ``most'' smooth Fano varieties have coregularity $0$. This expectation is confirmed in the two- and three-dimensional cases, see~\cite{ALP24}.

In this paper, we extend the notion of coregularity to the case of a variety $X$ with group action. We define the notion of $G$-coregularity, denoted by $\mathrm{coreg}_G(X)$. By definition, \emph{$G$-regularity} is
the maximum of~regularities of~all pairs that consist of $G$-invariant log-canonical complements on a given variety, see Definition~\ref{defin-regularity} for details. Similarly, we define \emph{$G$-coregularity} as the dual notion.
To~sum up, \mbox{$G$-coregularity} measures combinatorial complexity of $G$-invariant divisors in the (pluri-)anti-canonical linear systems.
To illustrate this notion, observe that while $\mathbb{P}^1$ has coregularity $0$, its $G$-coregularity depends on the finite group $G$.
%We start from the following warm-up result.

\begin{propositionint}[{see Example~\ref{example: coregularity of P1}}]
Let $G$ be a finite subgroup in $\operatorname{PGL}_2(\mathbb{K})=\operatorname{Aut}(\mathbb{P}^1)$. %, that is, $G\subset \mathrm{PGL}_2(\mathbb{K})$. According to Lemma \ref{lemma: action on P1},
%$G$ is isomorphic to one of the following groups: $\CG_n$, $\DG_{2n}$, $\AG_4$, $\SG_4$, or $\AG_5$.
Then $\coregG(\mathbb{P}^1)=0$ if and only if $G$ is cyclic or dihedral, and otherwise  $\coregG(\mathbb{P}^1)=1$.
\end{propositionint}

In this paper, we concentrate on the two-dimensional case. We study smooth two-dimensional Fano varieties, that is, smooth del Pezzo surfaces, as well as smooth rational surfaces that admit a~structure of a conic bundle. These two classes of surfaces naturally appear in the ($G$-equivariant) minimal model program for surfaces. The main invariant of a smooth del Pezzo surface $S$ is its degree~${d=(-K_S)^2}$. It is well known that $1\leqslant d\leqslant 9$. If $d=9$ then $S\cong\mathbb{P}^2$. In this case, we~have the~following result (the notation is explained in Theorem \ref{theorem: PGL3-classification}).

\begin{propositionint}[see Proposition \ref{proposition: dP9}]
\label{proposition: dP9-int}
Let $G$ be a finite subgroup in $\mathrm{PGL}_3(\mathbb{K})=\mathrm{Aut}(\mathbb{P}^2)$.
\begin{enumerate}
\item
If $G$ has type $(\mathrm{A})$, $(\mathrm{B}1)$, $(\mathrm{C})$,or $(\mathrm{D})$ then $\coregG(\mathbb{P}^2)=0$;
\item
if $G$ has type $(\mathrm{B}2)$, $(\mathrm{E})$, or $(\mathrm{H})$ then $\coregG(\mathbb{P}^2)=1$;
\item
if $G$ has type $(\mathrm{F})$, $(\mathrm{G})$, $(\mathrm{I})$, or $(\mathrm{K})$ then $\coregG(\mathbb{P}^2)=2$.
\end{enumerate}
\end{propositionint}

\iffalse
If $d=8$, then

\begin{propositionint}
\label{proposition: dP8-int}
Let $G$ be a finite subgroup in $\operatorname{Aut}(\mathbb{P}^1 \times \mathbb{P}^1)$, and $G_0$, $A_1$, and $A_2$ are defined as above.
\begin{enumerate}
\item
If each $A_i$ is either cyclic, or dihedral group, then $\coregG(\mathbb{P}^1 \times \mathbb{P}^1)=0$.
\item
If one of the groups $A_i$ is either cyclic, or dihedral, and the other is not, then $\coregG(\mathbb{P}^1 \times \mathbb{P}^1)=1$.
\item
If $A_1 \cong A_2 \cong A$, where $A$ is isomorphic to $\AG_4$, $\SG_4$, or $\AG_5$, and $G_0 \cong A \triangle_{A} A$ acts diagonally on $\mathbb{P}^1 \times \mathbb{P}^1$, then $\coregG(\mathbb{P}^1 \times \mathbb{P}^1)=1$.
\item
If each $A_i$ is isomorphic to $\AG_4$, $\SG_4$, or $\AG_5$, and the action of $G_0 \cong A_1 \triangle_R A_2$ on $\mathbb{P}^1 \times \mathbb{P}^1$ is not diagonal (in particular it holds if at least two groups among $A_1$, $A_2$, and $R$ are not isomorphic) then one has $\coregG(\mathbb{P}^1 \times \mathbb{P}^1)=2$.
\end{enumerate}
\end{propositionint}
\fi

For a smooth del Pezzo surface $S$ of degree $d\geqslant 6$ endowed with a faithful action of a finite group~$G$ (note that we do not assume that the surfaces are $G$-minimal), we obtain the classification of~groups~$G$ which have coregularity $0$, $1$, and $2$.
In the case when $d=8$ and $S$ is a quadric surface, that is, $S\cong\mathbb{P}^1\times\mathbb{P}^1$, see Proposition \ref{proposition: dP8}. In the case when $d=8$ and $S$ is isomorphic to a Hirzebruch surface~$\mathbb{F}_1$, see Proposition \ref{proposition: F1}. In the cases when $d=7$ or $d=6$, we always have $\mathrm{coreg}_G(S)=0$ for any finite group~${G\subset \mathrm{Aut}(S)}$ as shown in Proposition \ref{proposition: dP7} and Proposition \ref{proposition: dP6}.
The obtained results can be summarized as follows.

\begin{theoremintr}[see Theorem \ref{main-thm-del-pezzo-sec}]
\label{main-thm-del-pezzo}
Assume that $S$ is a smooth del Pezzo surface of degree $d\geqslant 6$, and $G$ is a finite subgroup in $\operatorname{Aut}(S)$. Then
\begin{enumerate}
\item
$\mathrm{coreg}_G(S)=0$ if and only if $G$ belongs to the normalizer of a torus $(\mathbb{K}^*)^2$ in $\mathrm{Aut}(S)$;
\item
if $\mathrm{coreg}_G(S)>0$ then
$\mathrm{coreg}_G(S)=1$ if and only if there exists a $G$-invariant curve $C$ on $S$ such that $-K_S-C$ is effective;
\item
$\mathrm{coreg}_G(S)=2$ otherwise.
\end{enumerate}
\end{theoremintr}

For $d\leqslant 5$, that is, for non-toric del Pezzo surfaces, the situation is more complicated. First of~all, the automorphism group of a smooth del Pezzo surface with $d\leqslant 5$ is finite.
The finite groups faithfully acting on smooth del Pezzo surfaces were classified in \cite{DI09}.
To look through all the possibilities for $G$ acting on a smooth del Pezzo surface with $d\leqslant 5$ seems to be a challenging task.

Our next result is the characterization of groups acting on a conic bundle with $G$-coregularity $0$. It turns out that the class of possible groups $G$ in this case is quite restrictive.

\begin{propositionint}[see Proposition~\ref{proposition: CB}]
Let $S$ be a smooth rational surface endowed with a faithful action of a finite group $G$. Assume that $S$ admits a structure of a $G$-equivariant conic bundle. Then there exists an exact sequence
\[
1\to G_F\to G\to G_B\to 1.
\]
If $\mathrm{coreg}_G(S)=~0$, then $G_F$ is either cyclic or dihedral, and $G_B$ is either cyclic or dihedral.
\end{propositionint}

If $X$ is a smooth variety and $D$ is a boundary $\mathbb{Q}$-divisor such that $K_X+D\sim_{\mathbb{Q}} 0$ and $(X, D)$ is log-canonical, then $(X, D)$ form a log
Calabi--Yau pair, see Section \ref{sect-log-CY}.
Log Calabi--Yau pairs of coregularity $0$ are important from the point of view of mirror symmetry,
see~\cite{KKP17},~\cite{Prz17}, \cite{HKP20},~\cite{GrossSiebert}.
In \cite{GHK15} it is proved that a log Calabi–Yau surface pair $(X, D)$ of coregularity zero with integral divisor $D$ is crepant equivalent (see Definition \ref{def:crep-bir-isom}) to a toric pair. This is no~longer the case in the $G$-equivariant setting, see Example \ref{ex-no-G-toric-model}. We also note that the theory of dual complexes of Calabi--Yau pairs was applied to the study of finite group actions on rationally connected varieties of dimension $3$ in \cite{Lo24}.

We explain the relation of $G$-coregularity with other invariants of varieties with a group action.
First of all, we note the relation of $G$-coregularity with $G$-log-canonical thresholds. Let us denote the $G$-equivariant log-canonical threshold of a $G$-variety $X$ by $\mathrm{lct}_G(X)$, see Section \ref{subsec-G-lct} for definition.
Using the result of Birkar \cite[Theorem 1.7]{B21}, in Proposition \ref{lem-G-lct-attained} we show that the $G$-log-canonical threshold is attained on a $G$-invariant divisor under the assumption that $\mathrm{lct}_G(X)\leq 1$. This allows us to prove the following criterion.

%\begin{propositionint}
%\end{propositionint}
%In arbitrary dimension, the following holds.

\begin{propositionint}[{see Proposition \ref{proposition: coregularity=lct}}]
Let $X$ be a klt Fano variety of dimension $n$, and let ${G\subset \mathrm{Aut}(X)}$ be a finite group.
%Then the inequality $\mathrm{lct}_G(X)>1$ implies $\mathrm{coreg}_G(X)=n$.
%Assume that a finite group $G$ faithfully acts on $\mathbb{P}^2$.
Then $\coregG(X)=n$ if and only if $\mathrm{lct}_G(X)>1$.
\end{propositionint}

This result and Proposition \ref{proposition: dP8} allows us to characterize finite subgroups ${G\subset \operatorname{Aut}(\mathbb{P}^1\times\mathbb{P}^1)}$ with $\mathrm{lct}_G(\mathbb{P}^1\times\mathbb{P}^1)>1$. In what follows, we denote the cyclic group by $\CG_n$, the dihedral group by $\DG_{2n}$, the~alternating group by $\AG_n$, and the~symmetric group by~$\SG_n$.

\begin{corollaryint}[{see Corollary \ref{corol-quadric-lct>1}}]
\label{corol-quadric-lct>1-int}
Let $C_1 \cong C_2 \cong \mathbb{P}^1$, $X \cong C_1 \times C_2$, and let $G\subset \operatorname{Aut}(X)$ be~a~finite group. Denote by $G_0 \subset G$ a subgroup of elements, acting trivially on $\operatorname{Pic}\left(X\right)$. Then the projections~${X \rightarrow C_i}$ are $G_0$-equivariant. These projections induce homomorphisms
$$
f_i \colon G_0 \rightarrow \operatorname{Aut}(C_i) \cong \mathrm{PGL}_2(\mathbb{K}).
$$
Denote the image $f_i(G_0)$ by $A_i$.

One has $\mathrm{lct}_G (\mathbb{P}^1 \times \mathbb{P}^1)>1$ if and only if each $A_i$ is isomorphic to $\AG_4$, $\SG_4$, or $\AG_5$, and the action of $G_0$ on $\mathbb{P}^1 \times \mathbb{P}^1$ is not diagonal (in particular it holds if at least two groups among $A_1$, $A_2$, and~$G_0$ are not isomorphic).
\end{corollaryint}

%The regularity measures how far a given Fano variety is from being exceptional. The latter means that any log-canonical complement is in fact Kawamata log terminal.
We also note the following relation between $G$-coregularity and exceptional quotient singularities with respect to the group $G$.
Recall that a singularity $x\in X$ is \emph{exceptional}, if there exists at~most one divisor over $X$ with log discrepancy $0$ with respect to any boundary $\mathbb{Q}$-divisor $D$ such that the~pair~$(X, D)$ is log-canonical, see Definition \ref{defin-exceptional}. A singularity $x\in X$ is called \emph{weakly exceptional}, if its plt blow-up is unique up to birational isomorphism, see Definition \ref{defin-weakly-exceptional}.
We have the following result.

\begin{propositionint}[{see Proposition~\ref{proposition: coregularity exceptional high dim}}]
\label{propint-lct}
Let $\overline{G} \subset \mathrm{SL}_n (\mathbb{K})$ be a finite subgroup.
%Assume that $n\leq 4$.
Let ${G}$ be the image of $\overline{G}$ in $\mathrm{PGL}_n(\mathbb{K})=\mathrm{Aut}(\mathbb{P}^{n-1})$ via the natural projection map $\mathrm{SL}_n (\mathbb{K})\to  \mathrm{PGL}_n (\mathbb{K})$. Then
\begin{enumerate}
\item
$\mathrm{lct}_G(\mathbb{P}^{n-1})>1$ if and only if the singularity $0\in\mathbb{A}^n/\overline{G}$ is exceptional. This condition is equivalent to the condition that $\coregG(\mathbb{P}^{n-1})=n-1$;
\item
$\mathrm{lct}_G(\mathbb{P}^{n-1})=1$ if and only if the singularity $0\in\mathbb{A}^n/\overline{G}$ is weakly exceptional and not exceptional. %This condition implies that $\coregG(\mathbb{P}^{n-1})= n-1$.
\end{enumerate}
\end{propositionint}

The first statement of Proposition \ref{propint-lct}(1) proves a conjecture proposed in \cite{ChSh11}, see Conjecture~\ref{conj-ChSh}.
%We note that the converse to the last part of the second statement of Proposition \ref{propint-lct}.(2) does not hold, see Example \ref{example: coregularity of P1} in the case when $G$ is a cyclic group.
In dimension~$2$, we have a more precise result (in which we use the notation as in~Theorem~\ref{theorem: PGL3-classification}).

\begin{propositionint}[{see Proposition \ref{prop-dim-2-exceptionality-criterion}}]
\label{prop-dim-2-exceptionality-criterion-int}
%Let $S$ be a projective plane endowed with a faithful action of a group $G$.
Let $\overline{G} \subset \mathrm{SL}_3 (\mathbb{K})$ be a finite subgroup.
%Assume that $n\leq 4$.
Let ${G}$ be the image of $\overline{G}$ in $\mathrm{PGL}_3(\mathbb{K})=\mathrm{Aut}(\mathbb{P}^{2})$ via the natural projection map $\mathrm{SL}_3(\mathbb{K})\to  \mathrm{PGL}_3 (\mathbb{K})$. Then
\begin{enumerate}
\item
%$\coregG(\mathbb{P}^2)=2$ if and only if
the singularity $0\in\mathbb{A}^3/\overline{G}$ is exceptional if and only if ${G}$ has type {\rm{(F)}}, {\rm{(G)}}, {\rm{(I)}}, or {\rm{(K)}};
\item
%$\coregG(\mathbb{P}^2)=1$ if and only if
the singularity $0\in\mathbb{A}^3/\overline{G}$ is weakly exceptional and not exceptional if and only if ${G}$ has type {\rm{(C)}} but is not isomorphic to $\CG_2^2\rtimes \CG_3$, or has type {\rm{(D)}} but is not isomorphic to $\CG_2^2\rtimes \SG_3$, or has type {\rm{(E)}};
\item
%$\coregG(\mathbb{P}^2)=0$ if and only if
the singularity $0\in\mathbb{A}^3/\overline{G}$ is not weakly exceptional if and only if ${G}$ has type {\rm{(A)}}, {\rm{(B)}}, {\rm{(H)}}, or is isomorphic to $\CG_2^2\rtimes \CG_3$ or $\CG_2^2\rtimes \SG_3$.
\end{enumerate}
\end{propositionint}

%In Question \ref{question-weakly-exc-coreg}, we ask whether it is possible to generalize Proposition \ref{prop-dim-2-exceptionality-criterion-int}(2) to arbitrary dimension, see also Remark \ref{rem-coreg-except}.

Let $X$ be a $G\mathbb{Q}$-Fano variety, see Section \ref{sec-mfs} for definition.
Recall that $X$ is called \emph{$G$-birationally rigid} if for any birational $G$-equivariant map $f\colon X\dashrightarrow X'$, where $X'$ is a $G$-Mori fiber space, $X'$ is a $G\mathbb{Q}$-Fano variety, and
there exists a $G$-birational automorphism $g\colon X\dashrightarrow X$ such that $f \circ g$ is a $G$-isomorphism.
Furthermore, if~the above assumptions imply that the map $f$ is an isomorphism then we say that~$X$ \mbox{is \emph{$G$-birationally}} \emph{super-rigid}.
Collecting results of \cite{Sa19} we obtain the following.
%We note there exists another version of the definition of birational rigidity where $g$ is assumed to be a $G$-\emph{birational} automorphism, which is used more frequenlty in the works devoted to the study of conjugacy of the finite subgroups of Cremona groups, cf. \cite{DI09}, \cite{Sa19}.

\begin{propositionint}[{see Proposition \ref{prop-dim-2-rigidity-criterion}}]
\label{prop-dim-2-rigidity-criterion-int}
%Let $S$ be a projective plane endowed with a faithful action of a group $G$.
Let $G \subset \operatorname{PGL}_3 (\mathbb{K}) =\operatorname{Aut}(\mathbb{P}^{2})$ be a finite subgroup. Then
\begin{enumerate}
\item
%$\coregG(\mathbb{P}^2)=2$ if and only if
$\mathbb{P}^2$ is $G$-birationally super-rigid if and only if ${G}$ has type {\rm{(F)}}, {\rm{(G)}}, {\rm{(I)}}, or {\rm{(K)}};
\item
%$\coregG(\mathbb{P}^2)=1$ if and only if
$\mathbb{P}^2$ is $G$-birationally rigid and not $G$-birationally super-rigid if and only if ${G}$ has type {\rm{(C)}} but is not isomorphic to $\CG_2^2\rtimes \CG_3$, or has type {\rm{(D)}} but is not isomorphic to $\CG_2^2\rtimes \SG_3$, or has type {\rm{(E)}}, or~type~{\rm{(H)}};
\item
%$\coregG(\mathbb{P}^2)=0$ if and only if
$\mathbb{P}^2$ is not $G$-birationally rigid if and only if ${G}$ has type {\rm{(A)}}, {\rm{(B)}}, or is isomorphic to $\CG_2^2\rtimes \CG_3$ or $\CG_2^2\rtimes \SG_3$.
\end{enumerate}
\end{propositionint}

Comparing Proposition \ref{proposition: dP9-int}, Proposition \ref{prop-dim-2-exceptionality-criterion-int}, and Proposition \ref{prop-dim-2-rigidity-criterion-int} and applying Proposition \ref{propint-lct} we obtain the following result.

\begin{corollaryint}
Let $\overline{G} \subset \mathrm{SL}_3 (\mathbb{K})$ be a finite subgroup.
%Assume that $n\leq 4$.
Let ${G}$ be the image of $\overline{G}$ in ${\mathrm{PGL}_3(\mathbb{K})=\mathrm{Aut}(\mathbb{P}^{2})}$ via the natural projection map $\mathrm{SL}_3(\mathbb{K})\to  \mathrm{PGL}_3 (\mathbb{K})$.
%Let $G\subset \mathrm{Aut}(\mathbb{P}^2)=\mathrm{PGL}_3(\mathbb{K})$ be a finite group.
Then the following conditions are equivalent:
\begin{enumerate}
\item
$\mathrm{coreg}_G(\mathbb{P}^2)=2$;
\item
$\mathrm{lct}_G(\mathbb{P}^{2})>1$;
\item
the singularity $0\in\mathbb{A}^3/\overline{G}$ is exceptional;
\item
$\mathbb{P}^2$ is $G$-birationally super-rigid.
\end{enumerate}
\end{corollaryint}

We summarize the obtained results together with what was already known on finite subgroups~$G$ in~$\mathrm{Aut}(\mathbb{P}^2)$ in the following table, in~which we use the notation as in Theorem \ref{theorem: PGL3-classification}.
\label{sec-the-table}

\begin{longtable}{|c|c|c|c|c|}

\hline Group & $G$-coregularity & $\mathrm{lct}_G$  & Exceptionality & $G$-rigidity \\

\hline A & $0$ & $\leqslant 1$ & no  & no   \\

\hline B1 & $0$ & $\leqslant 1$ & no  & no  \\

\hline B2 & $1$ & $\leqslant 1$ & no  & no  \\

\hline $\AG_4 \cong \CG_2^2 \rtimes \CG_3$ & $0$ & $\leqslant 1$ & no  & no  \\

\hline C but not $\AG_4$ & $0$ & $\leqslant 1$ & weakly & yes  \\

\hline $\SG_4 \cong \CG_2^2 \rtimes \SG_3$ & $0$ & $\leqslant 1$ & no & no  \\

\hline D but not $\SG_4$ & $0$ & $\leqslant 1$ & weakly & yes  \\

\hline E & $1$ & $\leqslant 1$ & weakly & yes  \\

\hline F & $2$ & $>1$  & yes & super  \\

\hline G & $2$ & $>1$  & yes  & super  \\

\hline H & $1$ & $\leqslant 1$ & no  & yes  \\

\hline I & $2$ & $\frac{4}{3}$  & yes & super  \\

\hline K & $2$ & $2$  & yes & super \\
\hline
\caption{Properties of finite subgroups $G\subset\mathrm{Aut}(\mathbb{P}^2$)}\label{table:P2}\\

\end{longtable}

Also, we introduce a more general notion of $G$-coregularity for (possibly infinite) groups acting on~pairs with linear systems, see Definition \ref{defin-regularity-linear-systems}. In Proposition~\ref{proposition: coregularity linear systems} we prove that this notion coincide with $G$-coregularity for a finite group $G$. However, this version of $G$-coregularity could be more suitable for infinite groups, see Example \ref{ex-with-infinite-group}.

\medskip

%We propose some questions.

The paper is organized as follows.
In Section \ref{sec-prelim} we collect some preliminary results and fix the~notation.
In Section \ref{sec-G-coreg}, we define $G$-coregularity and study its properties. We also study \mbox{$G$-log-canonical} thresholds, exceptional quotient singularities, and $G$-birational rigidity.
In Section \ref{sec-conic-bundles}, we~study $G$-coregularity of conic bundles and characterize the groups $G$ which have $G$-coregularity~$0$.
In~Section \ref{sec-del-pezzo-surfaces}, we study $G$-coregularity of smooth del Pezzo surfaces of degree at least $6$.
In Section \ref{sec-coreg-linear-systems} we define $G$-coregularity for possibly infinite groups acting on pairs with linear systems and study its properties.
Finally, in~Section \ref{sec-examples-questions} we provide some examples and formulate questions.

\medskip

\textbf{Acknowledgement.}
The article was prepared within the framework of the project ``International academic cooperation'' at HSE University.
The authors thank Jihao Liu and Constantin Shramov for useful conversations. The authors are grateful to anonymous referees for the careful reading of the manuscript and many helpful suggestions.

\section{Preliminaries}
\label{sec-prelim}
We work over an algebraically closed field $\mathbb{K}$ of characteristic $0$.
 All the varieties are
projective and defined over $\mathbb{K}$ unless  stated otherwise. We
use the language of the minimal model program (the MMP for short), see
e.g. \cite{KM98}.%, and the MMP for linear systems, cf. \cite{Al94}.

\subsection{Contractions} By a \emph{contraction} we mean a projective
morphism $f\colon X \to Y$ of normal varieties such that $f_*\OOO_X =
\OOO_Y$. In particular, $f$ is surjective and has connected fibers. A
\emph{fibration} is~defined as a contraction $f\colon X\to Y$ such
that $\dim Y<\dim X$.

Let $G$ be a group. By a \emph{$G$-variety} $X$ we mean a variety
$X$ together with a map $G\to \mathrm{Aut}(X)$ (not necessarily injective). If $f\colon X\to Y$ is a
$G$-equivariant contraction (resp., a $G$-equivariant fibration) of
$G$-varieties, we call it a \emph{$G$-contraction} (resp., a
\emph{$G$-fibration}).

\subsection{Pairs and singularities}
\label{subsec-types-of-singularities}
Let $X$ be a normal variety.
A $\mathbb{Q}$-divisor $B$ with
coefficients in $[0, 1]$ (resp., in $(-\infty, 1]$) is called a \emph{boundary} (resp., a \emph{sub-boundary}), if $K_X + B$ is $\mathbb{Q}$-Cartier.
In this case, $(X, B)$ is called a \emph{pair} (resp., a \emph{sub-pair}).
If $B$ is $G$-invariant, we call $(X, B)$ a \emph{$G$-pair} (resp., \mbox{a~\emph{$G$-sub-pair}}).
{A \emph{pair} $(X, B)$ \emph{over $Z$} is a pair $(X, B)$ together
with a projective morphism~${X\to Z}$}.
If $(X, B)$ is a $G$-pair and the morphism $X\to Z$ is $G$-equivariant, we call it a
\emph{$G$-pair over $Z$}.% (resp., a \emph{$G$-sub-pair over $Z$}).

For a birational contraction $f\colon W \to X$, by the \emph{log pullback} (or the \emph{crepant pullback}) of~the~sub-pair $(X, B)$ we mean the data $(W, B_W)$ such that
\[
K_W +B_W = f^*(K_X +B).
\]
We say that $f\colon W \to X$ is a
\emph{log resolution} of a sub-pair $(X,B)$ if $W$ is smooth, $\mathrm{Exc}(f)$ is a divisor, and the divisor $\mathrm{Exc}(f)\cup \mathrm{supp}(f^{-1}_*B)$ has simple normal crossings.

Let $\phi\colon W \to X$ be a
log resolution of $(X,B)$, and let $(W, B_W)$ be the log pullback of $(X, B)$.
The \emph{log discrepancy} of a prime divisor $D$ on $W$ with respect
to $(X, B)$ is the number
\[
a(D, X, B) = 1 - \mathrm{coeff}_D B_W.
\]
We
say that the pair $(X, B)$ is
\begin{itemize}
\item
\emph{lc} (resp., \emph{klt}) if $a(D, X, B)\geqslant 0$ (resp., $> 0$) for any log resolution~$\phi\colon W\to X$ and any divisor $D$ on $W$, %Note
%that if $(X, B)$ is $\epsilon$-lc,
%then automatically $\epsilon \leqslant 1$ because $a(D, X, B) = 1$ for
%almost all $D$.
\item
\emph{plt}, if $a(D, X, B)>0$ for any log resolution $\phi\colon W\to X$ and any $\phi$-exceptional divisor $D$ on~$W$,
\item
\emph{dlt}, if $a(D, X, B)>0$ for some log resolution~$\phi\colon W\to X$ and any $\phi$-exceptional divisor $D$ on $W$.
\end{itemize}

%A \emph{lc-place} of $(X, B)$ is a prime divisor $D$ over $X$ for which one has~\mbox{$a(D, X, B) = 0$}. A
%\emph{lc-center} is the image on $X$ of an lc-place.

%We say that a group $G$ acts on the pair $(X, B)$ if $X$ is a
%$G$-variety and the boundary divisor $D$ is $G$-invariant.
%\emph{Sub-pairs} and their singularities are defined similarly by
%letting the coefficients of $B$ to be any real number $\leqslant 1$. In
%this case we similarly have the notions of sub-lc, sub-klt, etc.

%\subsection{Minimal model program (MMP)}
%We will use standard results of the minimal model program (cf.
%\cite{KM98}). Assume $(X, B)$ is a pair over $Z$. Assume $H$ is an
%ample$/Z$ $\mathbb{Q}$-divisor and that $K_X + B + H$ is nef$/Z$.
%Suppose $(X, B)$ is klt or that it is $\mathbb{Q}$-factorial dlt. Then
%we can run an MMP$/Z$ on $K_X + B$ with scaling of $H$. If $(X, B)$ is
%klt and if either $K_X + B$ or $B$ is big$/Z$, then the MMP terminates
%\cite{BCHM10}. In this paper we mainly work in dimension 3 in which
%case termination is known in full generality for lc pairs.

\subsection{Log Fano and log Calabi--Yau pairs}
\label{sect-log-CY}
Let $(X, B)$ be an lc pair over $Z$. We say $(X,B)$ is~\emph{log Fano
over $Z$} if $-(K_X+B)$ is ample over $Z$. If $Z$ is a point then $(X,
B)$ is called a \emph{log Fano pair}. If~$B=0$ then $X$
is called a \emph{Fano variety over} $Z$. If $B=0$ and $Z$ is a point we say that $X$ is~a~\emph{Fano variety}.
A Fano variety of dimension $2$ is~called a \emph{del Pezzo surface}.

We say that an lc pair $(X,B)$ over $Z$ is \emph{log Calabi--Yau over $Z$} if $K_X + B
\sim_{\mathbb{Q}} 0$ over $Z$. If $Z$ is a point then $(X, B)$ is~called a \emph{log Calabi--Yau pair}.

\subsection{Mori fiber space}
\label{sec-mfs}
A $G\mathbb{Q}$-\emph{Mori fiber space} is a $G\mathbb{Q}$-factorial (that is, any $G$-invariant Weil divisor is Cartier, up to some multiple)
$G$-variety $X$ with at worst terminal singularities together with a
$G$-contraction $f\colon X\to Z$ to a variety $Z$ such that for the $G$-invariant relative Picard rank we have $\rho(X/Z)^G=1$,
and
$-K_X$ is ample over $Z$.
If $Z$ is a point, we say that $X$ is a $G\mathbb{Q}$-Fano variety.
If $X$ is $G$-factorial (that is, any $G$-invariant Weil divisor is Cartier, e.g. if $X$ is
smooth), we call it $G$-\emph{Mori fiber space}.

For a $G\mathbb{Q}$-Mori fiber space $f\colon X\to Z$,
if $\dim X=2$ and $\dim Z=0$, then we say that $X$ is a \emph{$G$-minimal del Pezzo surface}.
If $\dim X=2$ and $\dim Z=1$, then we say that $X$ is a \emph{$G$-minimal conic bundle}.
We reserve the terms \emph{del Pezzo $G$-surface} and \emph{$G$-conic bundle} for a del Pezzo surface and a conic bundle (that is, a Fano variety~$X$ over $\mathbb{P}^1$ such that $X$ is a smooth surface), respectively, equipped with an action of a group $G$, but~without the condition $\rho(X/Z)^G=1$. We note that in the case $\dim X=2$ we have that $X$ is smooth.

\subsection{Crepant $G$-birational equivalence}
\begin{definition}\label{def:crep-bir-isom}
{\em
Let $(X,D)$ and $(Y,D_Y)$ be two $G$-pairs.
We say that these two pairs
are {\em crepant $G$-birationally equivalent} (or \emph{$G$-crepant equivalent}) if there exists a $G$-equivariant commutative diagram
\begin{equation}
\begin{tikzcd}
& (V, D_V) \ar[rd, "\psi"] \ar[dl, swap, "\phi"] & \ \\
(X, D) \ar[rr, dashed, "\alpha"] & & (Y, D_Y)
\end{tikzcd}
\end{equation}
where $\alpha$ is a $G$-birational map, $\phi$ and $\psi$ are $G$-birational contractions, $\phi_*(D_V)=D_X$, ${\psi_*(D_V) = D_Y}$, and
\[
K_V + D_V=\phi^*(K_X+D) = \psi^*(K_Y + D_Y).
\]
Note that here $D_V$ is not necessarily a boundary since it may have negative coefficients.
In~the~previous setting, we write $(X,D)\sim_{G\text{-}\mathrm{cbir}} (Y,D_Y)$. If $G$ is the trivial group, we say that the pairs $(X, D)$ and $(Y, D_Y)$ are \emph{crepant equivalent},
and write $(X,D)\sim_{\mathrm{cbir}} (Y,D_Y)$.
}
\end{definition}

Note that the crepant equivalence is sometimes called K-equivalence in the literature, cf.~\cite{Wa98}.

\subsection{Dual complex}
\label{sec-dual-complex}
Let $D=\sum D_i$ be a Cartier divisor on a smooth variety $X$ where $D_i$ are irreducible components of $D$. Recall that $D$ has \emph{simple normal crossings} (snc for short), if all the components $D_i$ of $D$ are smooth, and any point in $D$ has an open neighborhood in the analytic topology that is analytically equivalent to~the~union of coordinate hyperplanes.

%\begin{defin}
%\label{def-dual-complex}
\emph{The dual complex}, denoted by $\mathcal{D}(D)$, of a simple normal crossing divisor $D=\sum_{i=1}^{r} D_i$ on~a~smooth variety $X$ is a regular CW-complex (more precisely, a $\Delta$-complex, so it admits a natural \mbox{PL-struc}\-ture) constructed as follows. %For any $i\in \{ 1, \ldots, r\}$ there exists a vertex $v_i$ in $\mathcal{D}(D)$ that corresponds to the divisor $D_i$.
The simplices $v_Z$ of $\mathcal{D}(D)$ are in bijection with irreducible components~$Z$ of~the~intersection $\bigcap_{i\in I} D_i$ for any subset $I\subset \{ 1, \ldots, r\}$, and the dimension of $v_Z$ is equal to~${\#I-1}$.
%For every subset $I\subset \{ 1, \ldots, r\}$ and every irreducible component $Z$ of the intersection $\bigcap_{i\in I} D_i$ there exists a simplex $v_Z$ in $\mathcal{D}(D)$ of dimension $\#I-1$ corresponding to $Z$.
%are in one-to-one correspondence with the irreducible components $D_i$ of $D$ and whose $m$-faces correspond bijectively to the irreducible components of the intersection of $m + 1$ irreducible components $D_{i_1}\cap \ldots \cap D_{i_{m+1}}$ for $i_1 < \ldots < i_{m+1}$ (they are also called strata of $D$). %The attaching maps are defined in the natural way.
%Then, the gluing maps are defined as follows.
%For every $I\subset \{ 1, \ldots, r\}$ and $j\in I$, there
The gluing maps are constructed as follows.
For any subset $I\subset \{ 1, \ldots, r\}$, let~${Z\subset \bigcap_{i\in I} D_i}$ be~any irreducible component, and for any $j\in I$ let $W$ be a unique component of~${\bigcap_{i\in I\setminus\{j\}} D_i}$ containing~$Z$. Then the gluing map is the inclusion of $v_W$ into $v_Z$ as a face of $v_Z$ that does not contain the vertex~$v_j$.

Note that the dimension of $\mathcal{D}(D)$ does not exceed $\dim X-1$. If $\mathcal{D}(D)$ is empty, we say that~${\dim \mathcal{D}(D)=-1}$. In what follows, for a divisor $D$ by $D^{=1}$ we denote the sum of the components of $D$ that have coefficient $1$ in it. For an lc sub-pair $(X, D)$, we define its dual complex $\mathcal{D}(X, D)$ as $\mathcal{D}(D_Y^{=1})$ where $f\colon (Y, D_Y)\to (X, D)$ is a log resolution of $(X, D)$, so the formula
\[
K_{Y} + D_Y= f^*(K_X + D)
\]
is satisfied. It is known that the PL-homeomorphism class of $\mathcal{D}(D_Y^{=1})$ does not depend on~the~choice of a log resolution%, so $\mathcal{D}(X, D)$ is well-defined
, see \cite{dFKX17}.
%\end{defin}
For more results on the topology of dual complexes of log Calabi--Yau pairs see \cite{KX16}. %In the case of log Fano pairs, see also \cite{Lo19} and \cite{LM20}.
%\subsection{$G$-coregularity}

\begin{remark}
\label{rem-group-acts-on-dual-complex}
Assume that $G$ is a finite group and $X$ is a $G$-variety. Let $(X, D)$ be an lc $G$-pair. Then $G$ acts on the dual complex $\mathcal{D}(X, D)$ by PL-homeomorphisms. Indeed, pick \mbox{a $G$-equivariant} log resolution $(Y, D_Y)$ of $(X, D)$. Then $\mathcal{D}(Y, D_Y)$ is a simplicial complex endowed with a simplicial action of $G$. Note that the action of $G$ on the topological space $\mathcal{D}(X, D)$ does not depend on~the~choice of a log resolution.%, since any two $G$-equivariant log resolutions of $G$ are $G$-equivariantly dominated by a third one.
\end{remark}

\begin{theorem}[{\cite[Proposition 5]{KX16}}]
\label{thm-kollar-xu}
Let $(S, D)$ be a log Calabi--Yau pair with $n=\dim S \leqslant 4$. Then the dual complex $\mathcal{D}(S,D)$ is either empty, or homeomorphic to the quotient of a sphere $\mathbb{S}^{k-1}/\Gamma$ where $\Gamma\subset \mathrm{O}_k(\mathbb{R})$ is a finite group and~$k\leqslant n$.
\end{theorem}

\begin{lemma}
\label{lemma: dual complex simple}
Let $S$ be a surface, and let $(S,D)$ be a log Calabi--Yau pair. The following holds.
\begin{itemize}
\item[(i)] The complex $\mathcal{D}(S,D)$ is homeomorphic either
to $\mathbb{S}^1$, or to a segment, or to two points, or~to~one point, or to an empty set.
\item[(ii)] If $D$ contains a component with coefficient less than $1$, then $\mathcal{D}(S,D)$ is not homeomorphic to a circle.
\end{itemize}
\end{lemma}
\begin{proof}
By~Theorem \ref{thm-kollar-xu}, the dual complex $\mathcal{D}(S,D)$ is either empty, or homeomorphic to~the~quotient of a sphere of dimension $0$ or $1$ by a finite group.
In the former case, we get either a point, or two points.
In the latter case, we have a quotient of $\mathbb{S}^1$ by a finite subgroup of $\mathrm{O}_{2}(\mathbb{R})$. Thus we~have that $\mathcal{D}(S,D)$ is homeomorphic either to a circle $\mathbb{S}^1$, or to a line segment. This proves the~first assertion.

%Let $(\widetilde{S}, \widetilde{D})$ be a log pullback of $(S, D)$ via a log resolution of $(S, D)$.
%The divisor $\widetilde{D}$ contains a component with coefficient less then $1$ if and only if $D$ does.

To establish the second assertion, we apply~\cite[Proposition 3]{KX16} to the pair $(S, D)$.
\end{proof}

We need the following result.
%The following proposition was proven in ????? without the action of a group. However, the proof is valid in the $G$-setting as well.
\begin{proposition}[{cf. \cite[16]{KX16}}]
\label{prop-disconnnected-case}
Let $(X, D)$ be a log Calabi--Yau pair where $X$ has dimension~$n$. Assume that $\mathcal{D}(X, D)$ is disconnected. Then we have %$(S, D)$ is $G$-crepant birational to a product
\[
(X, D)\sim_{\mathrm{cbir}}(Y, D_Y)\times (\mathbb{P}^1, \{0\}+\{\infty \})
\] where $(Y, D_Y)$ is a klt log Calabi--Yau pair. In particular, $\mathcal{D}(X, D)\cong \mathbb{S}^0$, and~${\coreg(X, D)=n-1}$.
\end{proposition}

%\subsection{Group actions. }
%We collect some generalities on the actions of finite groups on algebraic varieties.

\subsection{Automorphisms of projective spaces}

We will need the following.
\begin{lemma}[cf. {\cite[Lemma 4]{Po14}}]
\label{lem-faithful-action}
Assume that $G$ is a finite group such that $G\subset \mathrm{Aut}(X)$ where~$X$ is an algebraic variety.
%Assume that a finite group $G$ faithfully acts on an algebraic variety $X$,
Let $p\in X$ be a fixed point of the $G$-action.
Then the induced action of $G$ on the tangent space $T_p X$ is faithful.
\end{lemma}
%Let a group $G \subset \mathrm{GL}_n(\mathbb{K})$ act on an $n$-dimensional vector space $V$.

\begin{definition}[{\cite[Chapter I, \S14, \S44]{Bl17}}]
\label{definition: linear groups}
If the representation of the linear group $G$ on $V$ is~reducible then the group $G$ is called {\it intransitive}.
Otherwise the group $G$ is called {\it transitive}.

Let $G$ be a transitive group. If there exists a non-trivial decomposition $V = V_1 \oplus \dots \oplus V_l$ for~${l>1}$ to~subspaces such that for any element $g \in G$ and $i \in \{1, \dots, l\}$ one has \mbox{$gV_i = V_j$} for~some~${j \in \{1, \dots, l\}}$ then the group $G$ is called {\it imprimitive}. Otherwise the group $G$ is called {\it primitive}.

Let $f: \mathrm{SL}_n(\mathbb{K}) \rightarrow \mathrm{PGL}_n(\mathbb{K})$ be the natural surjective map. A group $G \subset \mathrm{PGL}_n(\mathbb{K})$ is called {\it intransitive} (resp. {\it imprimitive}, {\it primitive}, etc.) if so is the group $f^{-1}(G) \subset \mathrm{SL}_n(\mathbb{K})$.
%The linear group $G$ is \emph{irreducible} if for any non-zero vector
%$x\in V$, the $G$-orbit of $x$ spans $V$.
%An imprimitive irreducible group $G$ is \emph{monomial}, if the action of
%$G$ is induced from a 1-dimensional action of some subgroup $H\subset G$.
%In other words, $G=D\rtimes T$, where $D$ consists of diagonal matrices,
%and $T\subset S_n$ acts by permuting the basis of $V$.
\end{definition}

The following lemma is well-known.

\begin{lemma}
\label{lemma: representations over K}
Any representation of a finite group $G$ in $\mathrm{GL}_n(\mathbb{K})$ is conjugate to a representation of the group~$G$ in $\mathrm{GL}_n(\overline{\mathbb{Q}}) \subset \mathrm{GL}_n(\mathbb{K})$.
\end{lemma}

According to Lemma \ref{lemma: representations over K}, it is sufficient to know the classification of finite subgroups of~$\mathrm{PGL}_n(\mathbb{C})$ to classify finite subgroups of $\mathrm{PGL}_n(\mathbb{K})$.

The following result is well-known.
\begin{lemma}[{\cite[Chapter III]{Bl17}}]
\label{lemma: action on P1}
Assume that a finite group $G$ acts faithfully on $\mathbb{P}^1$. Then $G$ is isomorphic to a cyclic group $\CG_n$, a dihedral group $\DG_{2n}$, an~alternating group $\AG_4$, a~symmetric group~$\SG_4$, or an alternating group $\AG_5$.
\end{lemma}

Finite subgroups of $\mathrm{SL}_3(\mathbb{C})$ were completely classified in \cite[Chapter V]{Bl17} and \cite[Chapter~XII]{MBD16}. Applying these results and Lemma \ref{lemma: representations over K} we can obtain the following theorem.

\begin{theorem}
\label{theorem: PGL3-classification}
Any finite subgroup in $\mathrm{PGL}_3(\mathbb{K})$ modulo conjugation is one of the following.

\medskip

Intransitive groups:

\medskip

{\rm{(A)}} a diagonal abelian group;

\smallskip

{\rm{(B)}} a group having a unique fixed point on $\mathbb{P}^2_{\mathbb{K}}$.

\medskip

Imprimitive groups:

\medskip

{\rm{(C)}} a group $G$ having a normal diagonal abelian subgroup $N$ such that $G / N \cong \CG_3$;

\smallskip

{\rm{(D)}} a group $G$ having a normal diagonal abelian subgroup $N$ such that $G / N \cong \SG_3$.

\medskip

Primitive groups having a non-trivial normal subgroup (the Hessian group and its subgroups):

\medskip

{\rm{(E)}} the group $\CG_3^2 \rtimes \CG_4$ of order $36$;

\smallskip

{\rm{(F)}} the group $\CG_3^2 \rtimes Q_8$ of order $72$;

\smallskip

{\rm{(G)}} the Hessian group $\CG_3^2 \rtimes \mathrm{SL}_2(\mathbb{F}_3)$ of order $216$.

\medskip

Simple groups:

\medskip

{\rm{(H)}} the icosahedral group $\AG_5$ of order $60$;

\smallskip

{\rm{(I)}} the Klein group $\mathrm{PSL}_2(\mathbb{F}_7)$ of order $168$;

\smallskip

{\rm{(K)}} the Valentiner group $\AG_6$ of order $360$.
\end{theorem}

\begin{remark}
\label{remark: case B}
Let $G$ be any group of type $(\mathrm{B})$. Then $G$ is a subgroup of $\mathrm{GL}_2(\mathbb{K})$. Since $G$ has a unique fixed point on $\mathbb{P}^2$, it is non-abelian.  Such groups are described in \cite[Lemma 4.5]{DI09}. All these groups have an invariant line on $\mathbb{P}^2$, therefore there is a homomorphism $f \colon G \rightarrow \mathrm{PGL}_2(\mathbb{K})$. The group $f(G)$ cannot be cyclic by the assumption that $G$ has a unique fixed point on $\mathbb{P}^2$, so $f(G)$ is isomorphic to one of the following groups: $\DG_{2n}$, $\AG_4$, $\SG_4$, or $\AG_5$. If $f(G) \cong \DG_{2n}$ then for convenience we will say that $G$ has type $(\mathrm{B}1)$, and otherwise we will say that $G$ has type $(\mathrm{B}2)$.
\end{remark}

\section{$G$-coregularity}
\label{sec-G-coreg}
\subsection{Definitions and main properties}

The notion of regularity was introduced in \cite[7.9]{Sh00}, see also {\cite{Mo24}.
In this section, we introduce the notion of $G$-coregularity where $G$ is a group.
\begin{definition}
\label{defin-regularity}
Let $X$ be a normal projective $G$-variety with at worst klt singularities.. By the \emph{$G$-regularity} $\mathrm{reg}_G(X, D)$ of an lc $G$-pair $(X, D)$ we mean $\dim \mathcal{D}(X, D)$ where~$\mathcal{D}(X, D)$ is the dual complex of the pair defined in Section \ref{sec-dual-complex}. For an integer $l\geqslant 1$, we define the $l$-th \emph{complete $G$-regularity} of~$X$ by~the formula
\[
\mathrm{reg}_{G, l}(X) = \max \{ \mathrm{reg}_G(X, D)\ |\ D\in \frac{1}{l} |-lK_X|\}.
\]
In this situation, $K_X+D$ is called an \emph{$l$-complement} of $K_X$.
Then the \emph{complete $G$-regularity} of $X$ is
\[
\mathrm{reg}_G(X) = \max_{l\geqslant 1} \{\mathrm{reg}_{G, l}(X)\}.
\]
Note that $\mathrm{reg}_G(X)\in \{-1, 0,\ldots, \dim X-1\}$ where by convention we say that the dimension of~the~empty set is $-1$. The \emph{complete $G$-coregularity} of $X$ is defined as the number
$$
\mathrm{coreg}_G(X) = \dim X -1-\mathrm{reg}_G(X).
$$
Throughout the paper, for brevity we write $G$-regularity and $G$-coregularity of $X$ instead of complete $G$-regularity and $G$-coregularity of $X$. If the group $G$ is trivial, we use the terms regularity and coregularity and write $\mathrm{reg}(X, D)$, $\mathrm{reg}(X)$ and $\mathrm{coreg}(X)$.

\iffalse
Also, we define
\[
\mathrm{coreg}_{G, l}(X) = n-1-\mathrm{reg}_{G, l}(X). %,\quad \quad \quad \mathrm{coreg}_G(X) = n-1-\mathrm{reg}_G(X).
\]
\fi
\end{definition}

We make a remark on terminology.
The notion of regularity of~a~log-canonical pair $(X, D)$ on a normal projective variety $X$ was introduced by V. Shokurov in \cite[Proposition-Definition 7.11]{Sh00}.
There, the complete regularity of $(X, D)$ was defined as the~maximum of regularities of all pairs that consist of log-canonical complements of $K_X + D$.
In the case of pairs with zero boundary, complete regularity is called regularity of a variety in \cite{Mo24}, where the notion of coregularity was also defined.
%In our generalizetion to the $G$-setting, for simplicity we stick to the latter terminology.
\begin{remark}
We give a definition of $G$-regularity and $G$-coregularity for arbitrary groups. However in the paper we mostly consider the case of finite groups, cf. Section $5$.
\end{remark}

Clearly, one has $\mathrm{reg}_{G, l}(X)\leqslant \mathrm{reg}_{G, kl} (X)$ for any $k,l\geqslant1$. If $G$ is a trivial group, then these notions give us the classical notions of regularity and coregularity.

%\begin{remark}
%Since for a finite group $G$ and a $G$-variety $X$ any log resolution of $X$ is dominated by a $G$-equivariant resolution,
%one can replace a resolution by a $G$-equivariant resolution in Definition~\ref{defin-regularity}.
%\end{remark}

\begin{remark}
\label{rem-monotonicity}
Let $X$ be a $G$-variety.
Let $H\subset G$ be a subgroup. Then clearly
$$
\mathrm{coreg}_H(X)\leqslant \coregG(X).
$$
\end{remark}

\begin{remark}
One can formulate an analogue of Definition \ref{defin-regularity} in the case when the variety~$X$ is defined over a field $\mathbb{K}$ which is not algebraically closed, and the group $G$ is the Galois group of~a~field extension $\mathbb{K}\subset \mathbb{L}$. However, in this paper we concentrate on the geometric case.
\end{remark}

\begin{example}
\label{example: coregularity of P1}
Let $G$ be a finite group such that $G\subset\mathrm{Aut}(\mathbb{P}^1)=\mathrm{PGL}_2(\mathbb{K})$. By~Lemma~\ref{lemma: action on P1} the group $G$ is isomorphic to one of the following groups: $\CG_n$, $\DG_{2n}$, $\AG_4$, $\SG_4$, or $\AG_5$.
Note that in the case $G=\CG_n$, there exist two $G$-invariant points, and in the case $G=\DG_{2n}$ there is an orbit of cardinality $2$.
Since the minimal cardinality of a $G$-orbit is greater than $2$ for the other groups, it follows that $\coregG(\mathbb{P}^1)=0$ if and only if $G$ is cyclic or dihedral, that is, $G=\CG_n$ or $G=\DG_{2n}$, and otherwise $\coregG(\mathbb{P}^1)=1$.
\end{example}

\begin{example}
Let $G$ be any finite group. Put $N=|G|$.
Then there exists a $G$-variety $X$ of dimension $N-1$ such that $\coregG(X)=0$. Indeed, put $X=\mathbb{P}(V)$, where $V$ is a regular representation of $G$. Let $D_i$ be the coordinate hyperplanes in $\mathbb{P}(V)$ interchanged by the group $G$, and put $D=\sum_{i=1}^{N} D_i$. Then $G$-coregularity $0$ is attained on the pair $(X, D)$.
\end{example}

\subsection{Log-canonical thresholds}
\label{subsec-G-lct}
We recall the notion of the $G$-invariant log-canonical threshold. Let $G$ be a finite group such that $G\subset \mathrm{Aut}(X)$ where $X$ is normal projective variety.
Let $(X, D)$ be an lc pair.
Then the \emph{log-canonical threshold} of an effective $\mathbb{Q}$-Cartier $\mathbb{Q}$-divisor $B$ with respect to the pair $(X, D)$ is defined as
\[
\mathrm{lct} (X, D; B) = \mathrm{sup} \{ \lambda \in \mathbb{Q}\ |\ (X, D + \lambda B)\ \text{is lc} \}.
\]
The \emph{$G$-invariant log-canonical threshold}, or \emph{$G$-log-canonical threshold} for short, of an lc $G$-pair $(X, D)$ is defined as
\[
\mathrm{lct}_{G} (X, D) = \mathrm{inf} \{ \mathrm{lct}(X, D; B)\ |\  B\ \text{is a}\ G\text{-invariant $\mathbb{Q}$-divisor such that}\ K_X+D+B\sim_{\mathbb{Q}}0\}\in\mathbb{R}\cup \{\infty\}.
\]
If $G$ is a trivial group we write $\mathrm{lct}(X, D)$ instead of $\mathrm{lct}_G(X, D)$.
Then $\mathrm{lct}(X, D)$ is called the \emph{log-canonical threshold} of $(X, D)$.
If $D=0$ we write $\mathrm{lct}_G(X)$ instead of $\mathrm{lct}_G(X, D)$. If $G$ is~trivial and~$D=0$ we write $\mathrm{lct}(X)$ and call it the \emph{log-canonical threshold} of $X$. Note that in \cite{Ch08}, the author uses the notation $\mathrm{lct}(X,G)$ for $\mathrm{lct}_G(X)$.

We thank Jihao Liu for suggesting us the following result.

\begin{lemma}
\label{lem-G-lct-attained}
Let $G$ be a finite group such that $G\subset \mathrm{Aut}(X)$ where $X$ is a klt Fano variety.
Assume that $\mathrm{lct}_G(X)\leq 1$. Then the $G$-log-canonical threshold is attained on a $G$-invariant divisor~$B$ on $X$, that is,
\[
\mathrm{lct}_G(X) = \mathrm{lct}(X; B).
\]
\end{lemma}
\begin{proof}
Consider the quotient $\pi\colon X\to Y=X/G$. Note that $\pi$ is a finite morphism and $(Y, R)$ is~a~klt log Fano pair where $R$ is an (effective) ramification $\mathbb{Q}$-divisor of $\pi$ defined by the Hurwitz formula
\[
K_X = \pi^*(K_Y+R).
\]
%We claim that $\mathrm{lct}(Y, R)\leq 1$.
We have $\mathrm{lct}(Y, R)=\mathrm{lct}_G(X)\leq 1$, cf. \cite[Proposition 6.4]{Go20}.
%On the other hand, $\mathrm{lct}(Y, R) \geq \mathrm{lct}(Y)$. We conclude that
%\[
%1\geq \mathrm{lct}_G(X) = \mathrm{lct}(Y, R) \geq \mathrm{lct}(Y).
%\]
\iffalse
Indeed, since $\mathrm{lct}_G(X)\leq 1$, it follows that there exists a sequence of $G$-invariant boundaries $D_n$ on $X$ with $K_X+D_n\sim_{\mathbb{Q}} 0$ such that $\mathrm{lct}_G(X, D_n)=\lambda_n$, and $\lim_n \lambda_n \leq 1$. Consider a divisor $B_n$ on $Y$ such that
\[
K_X + D_n = \pi^*(K_Y + B_n)
\]
where the coefficients of $B_n$ are defined by the Hurwitz formula
\[
d_{i,j} = 1 - r_{i,j}(1-b_i).
\]
Here $r_{i,j}$ is the ramification index. Hence
\[
b_i = 1 - (1-d_{i,j})/r_{i,j}.
\]
In particular, $B_n$ is a boundary, and the pair $(Y, B_n)$ is lc.
\fi
%Put $B_n=\pi(D_n)$. Then $K_Y+B_n\sim_{\mathbb{Q}} 0$ and $\mathrm{lct}_G(X, D_n) = \mathrm{lct}(Y, B_n)=\lambda_n$. It follows that $\mathrm{lct}(Y)\leq 1$.
By \cite[Theorem 1.7]{B21}, using the~fact that $\mathrm{lct}(Y, R)\leq 1$, we know that the log-canonical threshold of $(Y, R)$ is attained on some $\mathbb{Q}$-divisor~$B_Y$ on $Y$ with $K_Y+R+B_Y\sim_{\mathbb{Q}} 0$:
\[
\mathrm{lct}(Y,R) = \mathrm{lct}(Y, R;B_Y).
\]
We conclude that $\mathrm{lct}_G(X)$ is attained on some $G$-invariant $\mathbb{Q}$-divisor $B$ on $X$ where $B$ is defined so~that~${\pi(B)=B_Y}$ and the Hurwitz formula
\[
0\sim_{\mathbb{Q}}K_X + B =\pi^*(K_Y+R+B_Y)
\]
holds.
%Put $D=\pi^*(B)$. Then we have that $K_X+D\sim_{\mathbb{Q}} 0$ and $\mathrm{lct}_G(X)=\mathrm{lct}(X, D)$, so the $G$-invariant log-canonical threshold of $X$ is attained on $D$.
\end{proof}

The following proposition relates the notion of $G$-log-canonical threshold with \mbox{$G$-core}\-gularity.
\begin{proposition}
\label{proposition: coregularity=lct}
Let $X$ be a klt Fano variety of dimension $n$, and let $G\subset \mathrm{Aut}(X)$ be a finite group.
Then
%\begin{itemize}
%\item
$\mathrm{lct}_G(X)>1$ if and only if $\mathrm{coreg}_G(X)=n$.
%\item
%if $\mathrm{lct}_G(X)<1$ then $\mathrm{coreg}_G(X)\leqslant n-1$.
%\end{itemize}
\end{proposition}
\begin{proof}
%Follows from the definitions.
%We prove the first assertion.
Assume that $\mathrm{lct}_G(X)>1$. Then for any $G$-invariant $\mathbb{Q}$-divisor $D$ such that $D\sim_{\mathbb{Q}} -K_X$ the pair $(X, D)$ is klt. In particular, the dual complex of the pair $(X, D)$ is empty. It follows that~$\mathrm{coreg}_G(X)=n$.

Assume that $\mathrm{lct}_G(X)\leq 1$. Then by Lemma \ref{lem-G-lct-attained} there exists a $G$-invariant divisor $D$ on $X$ such that $K_X+D\sim_{\mathbb{Q}} 0$ and $\mathrm{lct}_G(X)=\mathrm{lct}(X; D)$. In particular, the pair $(X, \lambda D+H)$ is lc and not klt for some $\lambda\leq 1$ and some general ample $G$-invariant $\mathbb{Q}$-divisor $H$ on $X$ such that $K_X+\lambda D + H\sim_{\mathbb{Q}} 0$.
It follows that the dual complex of $(X, \lambda D+H)$ is non-empty, and so $\mathrm{coreg}_G(X)<n$.
% Assume that $\mathrm{coreg}_G(X)=n$, and hence for any $G$-invariant $\mathbb{Q}$-divisor $D$ such that $D\sim_{\mathbb{Q}} -K_X$ the pair $(X, D)$ is klt.
\end{proof}

%We expect that the converse to the first claim of Proposition \ref{proposition: coregularity=lct} holds. However, to prove it, we need to know that the $G$-log-canonical threshold is attained at some divisor. To the best of our knowledge, it is not known yet.

\iffalse
Recall the following results.

\begin{theorem}[SSYLKA???]
Let $X$ be a Fano $G$-variety of dimension $n$ with klt singularities such that the inequality
\[
\mathrm{lct}_G(X)>n/(n+1)
\]
is satisfied. Then $X$ admits an orbifold K\"ahler–Einstein metric.
\end{theorem}

\begin{corollary}[ETO POCHEMU???]
Let $X$ be a Fano $G$-variety of dimension $n$ with klt singularities such that $\mathrm{coreg}_G(X)=n$. Then $X$ admits an orbifold K\"ahler–Einstein metric.
\end{corollary}
\fi

\begin{example}
We recall the following results on the $G$-log-canonical thresholds of del Pezzo surfaces, see \cite[Example 1.9, Example 6.5, Lemma 5.7, Example 1.11, Lemma~5.6]{Ch08}. Proposition \ref{proposition: coregularity=lct} implies that in the next cases we have $\coregG(X)=2$. One can see that the~corresponding $G$-actions on $\mathbb{P}^2$ are unique.

\begin{enumerate}
\item
Let $X=\mathbb{P}^2$, and $G=\operatorname{PGL}_2(7)$. Then $\operatorname{lct}_G(x)= \frac{4}{3}>1$.%, so $\mathrm{coreg}_G(X)=2$.

\item
Let $X=\mathbb{P}^2$, and $G=\AG_6$. Then $\operatorname{lct}_G(x)= 2>1$.%, so by Proposition~\ref{proposition: coregularity=lct} one has $\mathrm{coreg}_G(X)=2$.

\item
Let $X$ be a smooth del Pezzo surface of degree $5$, and $G=\AG_5$. Then $\operatorname{lct}_G(X)= 2>1$.%, so by Proposition~\ref{proposition: coregularity=lct} one has $\mathrm{coreg}_G(X)=2$.

\item
Let $X$ be the Clebsch cubic surface, that is, a cubic surface in $\mathbb{P}^3$ given by the equation
\[
x^2y + y^2z + z^2t + t^2x = 0.
\]
We have $\operatorname{Aut}(X)\cong \SG_5$.
%Put $G=\mathrm{Aut}(X)$.
Then $\mathrm{lct}_{\AG_5}(X)=\mathrm{lct}_{\SG_5}(X)=2>1$. %, so %by Proposition~\ref{proposition: coregularity=lct} one has $\mathrm{coreg}_{\AG_5}(X)=\mathrm{coreg}_{\SG_5}(X)=2$.

\item
Let $X$ be the Fermat cubic surface, which is a smooth del Pezzo surface of degree $3$ with~${G=\mathrm{Aut}(X)\cong \CG_3^3\rtimes \SG_4}$.
%Put $G=\mathrm{Aut}(X)$.
Then $\mathrm{lct}_{G}(X)= 4>1$. %so by Proposition~\ref{proposition: coregularity=lct} one has $\mathrm{coreg}_{G}(X)=2$.
\end{enumerate}
\end{example}

%As a consequence of Proposition \ref{proposition: coregularity=lct}, we have $\mathrm{coreg}_G(X)=2$ in the cases (1)-(5).

\begin{theorem}[{\cite[Theorem 1.13]{ChW13}}]
\label{thm-ChW13}
Let $X$ be a smooth del Pezzo surface of degree $d$.
Then there exists a~finite subgroup $G \subset \operatorname{Aut}(X)$ such that $\operatorname{lct}_G(X)> 1$ (and in particular, $\mathrm{coreg}_G(X)=2$) if and only if one of the following holds:
\begin{enumerate}
\item
$X=\mathbb{P}^2$, furthermore in this case $G$ has type $(\mathrm{F})$, $(\mathrm{G})$, $(\mathrm{I})$ or $(\mathrm{K})$ as in Proposition \ref{theorem: PGL3-classification},
\item
$X=\mathbb{P}^1\times\mathbb{P}^1$,
\item
$\operatorname{Aut}(X)$ is finite and one of the following holds:
\begin{enumerate}
\item
$d=1$ and $\operatorname{Aut}(X)$ is not abelian,
\item
$d=2$ and $\operatorname{Aut}(X)\in \{ \SG_4\times \CG_2, (\CG_4^2\rtimes \SG_3)\times \CG_2, \operatorname{PSL}_2(7)\times\CG_2\}$,
\item
$d=3$ and $X$ is the Clebsch cubic surface or the Fermat cubic surface,
\item
$d=4$ and $\operatorname{Aut}(X)\in\{ \CG_2^4\rtimes \SG_3, \CG_2^4\rtimes \DG_{10}\}$,
\item
$d=5$.
%\item
%$d=5$.
\end{enumerate}
\end{enumerate}
\end{theorem}

In the case $X=\mathbb{P}^2$, Theorem \ref{thm-ChW13} follows from Proposition \ref{proposition: coregularity exceptional high dim} and Proposition \ref{proposition: dP9} which we~prove using a different approach. In the case $X=\mathbb{P}^1\times\mathbb{P}^1$, we improve the result of Theorem~\ref{thm-ChW13} and obtain a complete list of finite groups for which $\mathrm{lct}_G(\mathbb{P}^1\times\mathbb{P}^1)>1$, see Proposition \ref{proposition: dP8}.

\subsection{Exceptional singularities}

\begin{definition}
\label{defin-exceptional}
Let $P \in X$ be a germ of a klt singularity and let $(X, D)$ be an lc pair. Then $(X, D)$ is called \emph{exceptional} if there exists at most one exceptional divisor~$E$ over $X$ with $a(E,X,D)=0$. A germ of a klt singularity $P \in X$ is said to~be~\emph{exceptional} if $(X, D)$ is exceptional for any $D$ such that the pair $(X, D)$ is lc.
\end{definition}

The notion of a plt blow up was introduced by Yu. Prokhorov in \cite{Pr00}.
Let $X$ be a normal algebraic variety and let $f\colon Y \to X$ be a birational morphism such that the exceptional locus of~$f$ contains exactly one irreducible divisor $E$. Then $f\colon (Y, E) \to X$ is called a \emph{plt blow-up} of $X$ if~$(Y,E)$ is plt and $-(K_Y +E)$ is $f$-ample.

Note that if $X$ is $\mathbb{Q}$-factorial and has klt singularities then the condition that $-(K_Y +E)$ \mbox{is~$f$-ample} holds automatically under the assumption that $(Y,E)$ is plt.

\begin{definition}
\label{defin-weakly-exceptional}
Let $P \in X$ be a germ of a klt singularity.
The singularity is said to be \emph{weakly exceptional} if its plt blowup is unique, up to birational equivalence.
\end{definition}

If $P \in X$ is exceptional, then $P \in X$ is weakly exceptional, cf.~\cite[Proposition~2.7]{MP99}, \cite[2.6]{PS01}
%\cite{Pr00}.

\begin{example}[{see~\cite[1.5]{Sh00}~\cite[Example 4.7]{Pr00}}]
Let $P\in X$ be a $2$-dimensional quotient singularity, that is, $X$ is isomorphic to $\mathbb{A}^2/\overline{G}$ where $\overline{G}\subset \mathrm{GL}_2(\mathbb{K})$ is a finite group. Then
\begin{itemize}
\item
$P\in X$ is exceptional if and only if it has type $E_6, E_7, E_8$,
\item
$P\in X$ is weakly exceptional and not exceptional if and only if it has type $D_n$,
\item
$P\in X$ is not weakly exceptional if and only if it has type $A_n$.
\end{itemize}
\end{example}

We will use the following notation.
For a finite subgroup $\overline{G} \subset \mathrm{SL}_n (\mathbb{K})$, we denote by ${G}$ its image in $\mathrm{PGL}_n(\mathbb{K})$ via the natural projection map $\mathrm{SL}_n (\mathbb{K})\to  \mathrm{PGL}_n (\mathbb{K})$.

In what follows, we use the notation of Theorem \ref{theorem: PGL3-classification}.

\begin{theorem}[{\cite[Theorem 3.13]{MP99}}]
\label{thm-exc-dim2}
The singularity $0\in\mathbb{A}^3/\overline{G}$ is exceptional if and only if~${G}$ has type {\rm{(F)}}, {\rm{(G)}}, {\rm{(I)}}, or {\rm{(K)}}.
\end{theorem}

%We collect some results on exceptional and weakly exceptional quotient singularities.
\iffalse
\begin{theorem}[{\cite[Proof of Theorem 1.2]{MP99}}]
%Let $\overline{G}\subset SL_3(\mathbb{K})$ be a finite group.
%Let ${G}$ be the image of $\overline{G}$ in $\mathrm{PGL}_3(\mathbb{K})$.
The singularity $0\in\mathbb{A}^3/\overline{G}$ is exceptional
if and only if $G$ is isomorphic either to $\AG_6$, or to the Klein group $\mathrm{PSL}_2(\mathbb{F}_7)$ of order $168$
or to the Hessian group $\CG_3^2 \rtimes \mathrm{SL}_2(\mathbb{F}_3)$ of order $216$,
or its triple cover, complex reflection group of order $648$.
\end{theorem}
\fi
\begin{theorem}[{\cite[Theorem 1.14]{Sa12}}]
\label{thm-weakly-exc-dim2}
%Let $\overline{G}\subset SL_3(\mathbb{K})$ be a finite group.
The singularity $0\in\mathbb{A}^3/\overline{G}$ is weakly exceptional but not~exceptional if and only if ${G}$ has type {\rm{(C)}}, but is not~isomorphic to $\CG_2^2\rtimes \CG_3$, or has type {\rm{(D)}} but is not~isomorphic to $\CG_2^2\rtimes \SG_3$, or type {\rm{(E)}}.
%$\overline{G}$ is an imprimitive group, and $G$ is not isomorphic to $\CG_2^2\rtimes \CG_3$ or to
%$\CG_2^2\rtimes \SG_3$,
%or if $\overline{G}$ is isomorphic to the index $2$ subgroup of the Hessian group $\CG_3^2 \rtimes \mathrm{SL}_2(\mathbb{F}_3)$ of order $216$.

%$(\mathrm{B}2)$, $(\mathrm{E})$ or $(\mathrm{H})$
\end{theorem}

We note that in Theorem \ref{thm-exc-dim2} and Theorem \ref{thm-weakly-exc-dim2} it is important that $\overline{G}$ is assumed to be a subgroup of $\mathrm{SL}_3(\mathbb{K})$.
Recall that a non-trivial element $\overline{g}\in\overline{G}\subset \mathrm{GL}_n (\mathbb{K})$ is called a reflection (or sometimes a quasi-reflection) if there is a hyperplane in $\mathbb{C}^n$ that is fixed pointwisely by $\overline{g}$. %its image $g\in G$ via the natural map~${\mathrm{GL}_n (\mathbb{K})\to  \mathrm{PGL}_n (\mathbb{K})}$.
Note that if $G\subset \mathrm{SL}_n(\mathbb{K})$ then $G$ does not contain reflections.
%Recall that an element $g\in G$ is called a reflection (or sometimes a quasi-reflection) if there is a hyperplane in $\mathbb{P}^n$ that is fixed pointwise by $\overline{g}\in\overline{G}$.
In~\cite{ChSh11}, the following conjecture was proposed:

\begin{conjecture}[{\cite[Conjecture 1.23]{ChSh11}}]
\label{conj-ChSh}
Let $\overline{G} \subset \mathrm{GL}_n (\mathbb{K})$ be a finite subgroup that does not contain reflections.
%Assume that $n\leq 4$.
Let ${G}$ be the image of $\overline{G}$ in $\mathrm{PGL}_n(\mathbb{K})$ via the natural map~${\mathrm{GL}_n (\mathbb{K})\to  \mathrm{PGL}_n (\mathbb{K})}$. Then
$\mathrm{lct}_G(\mathbb{P}^{n-1})>1$ if and only if the singularity $0\in\mathbb{A}^n/\overline{G}$ is exceptional.
\end{conjecture}

\begin{proposition}
\label{prop-conj-is-true}
Conjecture \ref{conj-ChSh} holds true.
\end{proposition}
\begin{proof}
By \cite[Remark 1.16]{ChSh11} we may assume that $\overline{G}\subset \mathrm{SL}_n(\mathbb{K})$.
By \cite[Theorem~3.17]{ChSh11} we have that $\mathbb{A}^n/\overline{G}$ is exceptional if and only if the following holds: %$\mathrm{coreg}_{{G}}(\mathbb{P}^{n-1})=n-1$.
%In other words,
for any ${G}$-invariant divisor~$D$ on~$\mathbb{P}^{n-1}$ if the pair $(\mathbb{P}^{n-1}, D)$ is lc then it is klt. However, this condition is equivalent to~${\mathrm{lct}_{{G}}(\mathbb{P}^{n-1})>1}$ by Lemma \ref{lem-G-lct-attained}.
\end{proof}

\begin{proposition}
\label{proposition: coregularity exceptional high dim}
Let $\overline{G} \subset \mathrm{SL}_n (\mathbb{K})$ be a finite group.
%Assume that $n\leq 4$.
Let ${G}$ be the image of $\overline{G}$ in
$$
{\mathrm{PGL}_n(\mathbb{K})=\mathrm{Aut}(\mathbb{P}^{n-1})}
$$
via the natural projection map $\mathrm{SL}_n (\mathbb{K})\to  \mathrm{PGL}_n (\mathbb{K})$. Then
\begin{enumerate}
\item
$\mathrm{lct}_G(\mathbb{P}^{n-1})>1$ if and only if the singularity $0\in\mathbb{A}^n/\overline{G}$ is exceptional. This condition is equivalent to the condition that $\coregG(\mathbb{P}^{n-1})=n-1$;
\item
$\mathrm{lct}_G(\mathbb{P}^{n-1})=1$ if and only if the singularity $0\in\mathbb{A}^n/\overline{G}$ is weakly exceptional and not exceptional.%This condition implies that $\coregG(\mathbb{P}^{n-1})= n-1$.
\end{enumerate}
\end{proposition}
\begin{proof}
Note that since $\overline{G} \subset \mathrm{SL}_n (\mathbb{K})$, it does not contain reflections.
Then the first claim follows from Proposition \ref{prop-conj-is-true} combined with Proposition \ref{proposition: coregularity=lct}.
The second claim follows from Proposition \ref{prop-conj-is-true} and \cite[Theorem 3.16]{ChSh11}.% it follows that $\mathrm{lct}_G(\mathbb{P}^{n-1})=1$ is equivalent to the condition that the singularity $\mathbb{A}^n/\overline{G}$ is weakly exceptional and not exceptional.
% from Proposition \ref{prop-weakly-exceptional-criterion}.
%We prove the second claim. By the first claim and \cite[Theorem 3.16]{ChSh11} it follows that $\mathrm{lct}_G(\mathbb{P}^{n-1})=1$ is equivalent to the condition that the singularity $\mathbb{A}^n/\overline{G}$ is weakly exceptional and not exceptional. The
%Suppose that $\mathrm{lct}_G(\mathbb{P}^{n-1})=1$. Then clearly $\coregG(\mathbb{P}^{n-1})\leq n-1$. Assume that $\coregG(\mathbb{P}^{n-1})\leq n-2$ which is equivalent to $\mathrm{reg}_G(\mathbb{P}^{n-1})\geq 1$. Hence there exists at least two $G$-invariant divisors $E_1$ and $E_2$ over $\mathbb{P}^{n-1}$ with log-discrepancy $0$.
\end{proof}

In dimension $2$, we have a more precise result.

\begin{proposition}
\label{prop-dim-2-exceptionality-criterion}
%Let $S$ be a projective plane endowed with a faithful action of a group $G$.
Let $\overline{G} \subset \mathrm{SL}_3 (\mathbb{K})$ be a finite group.
%Assume that $n\leq 4$.
Let ${G}$ be the image of $\overline{G}$ in
$$
{\mathrm{PGL}_3(\mathbb{K})=\mathrm{Aut}(\mathbb{P}^{2})}
$$
via the natural projection map $\mathrm{SL}_3(\mathbb{K})\to  \mathrm{PGL}_3 (\mathbb{K})$. Then
\begin{enumerate}
\item
%$\coregG(\mathbb{P}^2)=2$ if and only if
the singularity $0\in\mathbb{A}^3/\overline{G}$ is exceptional if and only if ${G}$ has type {\rm{(F)}}, {\rm{(G)}}, {\rm{(I)}}, or {\rm{(K)}};
\item
%$\coregG(\mathbb{P}^2)=1$ if and only if
the singularity $0\in\mathbb{A}^3/\overline{G}$ is weakly exceptional and not exceptional if and only if ${G}$ has type {\rm{(C)}} but is not isomorphic to $\CG_2^2\rtimes \CG_3$, or has type {\rm{(D)}} but is not isomorphic to $\CG_2^2\rtimes \SG_3$, or has type~{\rm{(E)}};
\item
%$\coregG(\mathbb{P}^2)=0$ if and only if
the singularity $0\in\mathbb{A}^3/\overline{G}$ is not weakly exceptional if and only if ${G}$ has type {\rm{(A)}}, {\rm{(B)}}, {\rm{(H)}}, or~is isomorphic to $\CG_2^2\rtimes \CG_3$ or $\CG_2^2\rtimes \SG_3$.
\end{enumerate}
\end{proposition}
\begin{proof}
The first claim follows from Theorem \ref{thm-exc-dim2}. The second claim follows from Theorem \ref{thm-weakly-exc-dim2}. Finally, the third claim follows from Theorem \ref{theorem: PGL3-classification} and the fact that $\mathrm{coreg}_G(\mathbb{P}^2)\in\{0,1,2\}$.
\end{proof}

\subsection{Rigidity}

\begin{definition}
Let $X$ be a $G\mathbb{Q}$-Fano variety, see Section \ref{sec-mfs} for definition.
Recall that $X$ is called \emph{$G$-birationally rigid} if for any birational $G$-equivariant map $f\colon X\dashrightarrow X'$, where $X'$ is a $G$-Mori fiber space, $X'$ is a $G\mathbb{Q}$-Fano variety, and
there exists a $G$-birational automorphism $g\colon X\dashrightarrow X$ such that $f \circ g$ is a $G$-isomorphism.
Furthermore, if~the above assumptions imply that the map $f$ is an isomorphism then we say that~$X$ \mbox{is \emph{$G$-birationally}} \emph{super-rigid}.
%%%Let $X$ be a $G$-Mori fiber space.
%minimal del Pezzo surface with $\rho(X)^G = 1$ or a relatively $G$-minimal conic bundle.
%%%Then $X$ is called \emph{$G$-birationally rigid} if for any birational $G$-equivariant map $f\colon X\dashrightarrow X'$, where $X'$ is a $G$-Mori fiber space,
%is a $G$-minimal del Pezzo surface $X'$ with $\rho(X')^G = 1$ or a relatively $G$-minimal conic bundle,
%%%there exists a~$G$-birational automorphism $g\colon X\dashrightarrow X$ such that $f \circ g$ is a $G$-isomorphism.
%%%Furthermore, if the above assumptions imply that the map $f$ is an isomorphism then we say that~$X$ is \emph{$G$-birationally super-rigid}.
\end{definition}

%We note there exists another version of the definition of birational rigidity where $g$ is assumed to be a $G$-\emph{birational} automorphism, which is used more frequenlty in the works devoted to the study of conjugacy of the finite subgroups of Cremona groups, cf. \cite{DI09}, \cite{Sa19}.

%\begin{proposition}
%Let $G\subset \mathrm{Aut}(\mathbb{P}^2)=\mathrm{PGL}_3(\mathbb{K})$ be a finite group. Then $\mathbb{P}^2$ is $G$-birationally super-rigid if and only if $\mathrm{coreg}_G(\mathbb{P}^2)=2$.
%\end{proposition}

For finite subgroups of $\operatorname{PGL}_3(\mathbb{K}) = \operatorname{Aut} (\mathbb{P}^2)$ birational rigidity is studied in \cite{Sa19}. The main result is the following.

\begin{theorem}[{\cite[Theorem 1.3]{Sa19}}]
\label{prop-dim-2-rigidity-semicriterion}
The projective plane $\mathbb{P}^2$ is $G$-birationally rigid if and only if the action of $G$ is transitive and $G$ is not isomorphic to $\AG_4$ or $\SG_4$.
\end{theorem}

Note that in the notation of Theorem \ref{theorem: PGL3-classification} intransitive groups have types {\rm{(A)}} and {\rm{(B)}}, and groups of the other types are transitive. Moreover, the groups $\AG_4 \cong \CG_2^2\rtimes \CG_3$ and $\SG_4 \cong \CG_2^2\rtimes \SG_3$ have types {\rm{(C)}} and {\rm{(D)}} respectively (actually these groups are special cases of types {\rm{(C)}} and {\rm{(D)}} mentioned in Proposition \ref{prop-dim-2-exceptionality-criterion}).

Therefore we can say that for a finite group $G \subset \operatorname{PGL}_3 (\mathbb{K})$ the projective plane is \mbox{not $G$-bira}\-tionally rigid if and only if $G$ has type {\rm{(A)}}, {\rm{(B)}} or is isomorphic to $\CG_2^2\rtimes \CG_3$ or $\CG_2^2\rtimes \SG_3$.

The results of \cite{Sa19} also allow one to classify finite subgroups of $\operatorname{PGL}_3 (\mathbb{K})$ such that $\mathbb{P}^2$ \mbox{is $G$-birationally} super-rigid, and obtain the following.

\begin{proposition}
\label{prop-dim-2-rigidity-criterion}
%Let $S$ be a projective plane endowed with a faithful action of a group $G$.
Let $G \subset \operatorname{PGL}_3 (\mathbb{K}) =\operatorname{Aut}(\mathbb{P}^{2})$ be a finite group. Then
\begin{enumerate}
\item
%$\coregG(\mathbb{P}^2)=2$ if and only if
$\mathbb{P}^2$ is $G$-birationally super-rigid if and only if ${G}$ has type {\rm{(F)}}, {\rm{(G)}}, {\rm{(I)}}, or {\rm{(K)}};
\item
%$\coregG(\mathbb{P}^2)=1$ if and only if
$\mathbb{P}^2$ is $G$-birationally rigid and not $G$-birationally super-rigid if and only if ${G}$ has type {\rm{(C)}} but is not isomorphic to $\CG_2^2\rtimes \CG_3$, or has type {\rm{(D)}} but is not isomorphic to $\CG_2^2\rtimes \SG_3$, or has type {\rm{(E)}}, or~type~{\rm{(H)}};
\item
%$\coregG(\mathbb{P}^2)=0$ if and only if
$\mathbb{P}^2$ is not $G$-birationally rigid if and only if ${G}$ has type {\rm{(A)}}, {\rm{(B)}}, or is isomorphic to $\CG_2^2\rtimes \CG_3$ or $\CG_2^2\rtimes \SG_3$.
\end{enumerate}
\end{proposition}

\begin{proof}
Assertion $(3)$ follows from Theorem \ref{prop-dim-2-rigidity-semicriterion}. Therefore we need to find for which types \mbox{of~groups {\rm{(C)}}--{\rm{(K)}}} the projective plane is $G$-birationally super-rigid.

A $G$-minimal surface is $G$-birationally super-rigid if and only if one cannot construct a~$G$-equivariant Sarkisov link from this surface (cf. \cite[Theorem 7.7]{DI09}). Moreover, one can construct \mbox{a~$G$-equi}\-variant Sarkisov link from $\mathbb{P}^2$ if and only if $G$ has an orbit on $\mathbb{P}^2$ consisting of up to $8$ points in~general position.

Any group of type {\rm{(C)}} and {\rm{(D)}} has an orbit consisting of three non-collinear fixed points of $N$ (see Theorem \ref{theorem: PGL3-classification}) and permuted by $G / N$, since $N$ is a diagonalizable abelian group. Groups of~type~{\rm{(E)}} and {\rm{(H)}} have orbits consisting of six points in general position (see \cite[Propositions~3.16 and~3.17]{Sa19}). Therefore for a group $G$ of type {\rm{(C)}}, {\rm{(D)}}, {\rm{(E)}}, or {\rm{(H)}} the projective plane is \mbox{not $G$-birationally} super-rigid. This implies assertion $(2)$.

Groups of type {\rm{(F)}}, {\rm{(G)}}, {\rm{(I)}}, and {\rm{(K)}} do not have orbits consisting of up to $8$ points (see \cite[Propositions 3.16 and 3.19]{Sa19}). Therefore for a group $G$ of type {\rm{(F)}}, {\rm{(G)}}, {\rm{(I)}}, or {\rm{(K)}} the~projective plane is $G$-birationally super-rigid. This implies assertion $(1)$.
\end{proof}

\iffalse
\begin{remark}
\label{rem-coreg-except}
We note that if $\mathrm{coreg}_G(\mathbb{P}^{n-1})=n-2$, then it is not true in general that $\mathbb{A}^n/\overline{G}$ is weakly exceptional but not exceptional. Indeed, see Example \ref{example: coregularity of P1} in the case when $G$ is a cyclic group.
We also note that $\mathrm{lct}_G(X)=1$ does not imply $\coregG(X)=\dim X-1$ for arbitrary Fano variety~$X$. Indeed, let $X$ a general smooth del Pezzo surface of degree $1$, and let $G$ be the trivial group. Then $\mathrm{lct}_G(X)=\mathrm{lct}(X)=1$, see \cite{Ch08}, while $\mathrm{coreg}_G(X)=\mathrm{coreg}(X)=0$, see \cite[Proposition 2.3]{ALP24}.
However, we do not know whether the assumption that $\mathbb{A}^n/\overline{G}$ is weakly exceptional but not exceptional implies that $\mathrm{coreg}_G(\mathbb{P}^{n-1})=n-2$. We formulate it as Question \ref{question-weakly-exc-coreg}.
\end{remark}
\fi

\subsection{Properties of $G$-coregularity}

\begin{definition}
Let $(X,D)$ be an lc $G$-pair %Then $G$ has a map to
such that $\mathcal{D}(X,D)\neq \varnothing$.
Let $(\widetilde{X},\widetilde{D})$ be the crepant pullback of $(X, D)$ via a $G$-equivariant log resolution of $(X,D)$. Denote by $N$ the subgroup of $G$, consisting of all elements preserving each component of $\widetilde{D}$. Then we have the following exact sequence:
\begin{equation}
\label{exact-sequence-two-parts-of-G}
1 \rightarrow N \rightarrow G \rightarrow G_{\mathcal{D}} \rightarrow 1,
\end{equation}
where the group $G_{\mathcal{D}} \cong G / N$ acts faithfully on the topological space $\mathcal{D}(X,D)$ by PL-homeomor\-phisms, cf. Remark \ref{rem-group-acts-on-dual-complex}.
Since the topological space $\mathcal{D}(X,D)$ does not depend on the choice of~a~log resolution of $(X,D)$, it follows that the groups $N$ and $G_\mathcal{D}$ do not depend on the choice of~a~log resolution of $(X,D)$.
%%%We call the group $G_{\mathcal{D}}$ the \emph{combinatorial part} of $G$, and the group $N$ the~\emph{$\mathcal{D}$-trivial part} of the group $G$.
%permutes components of $\widetilde{D}$. We call the subgroup $N$.
\end{definition}

\begin{proposition}
\label{prop-coreg-0-arbitrary-dimension}
Consider a finite group $G$ and a $G$-variety $X$ of dimension $n$. Assume that~\mbox{$\coregG(X)=0$}.
Then the group $N$ defined in \eqref{exact-sequence-two-parts-of-G} has a fixed point on $X$. Moreover, $N$ is an abelian group of rank at most $n$.
\end{proposition}
\begin{proof}
Let $(X, D)$ be a log Calabi--Yau pair such that $\mathrm{coreg}_G(X,D)=0$. % coregularity $0$ is attained on $(X, D)$.
Let~$(\widetilde{X}, \widetilde{D})$ be a crepant pullback of $(X, D)$ via a $G$-equivariant log resolution. So we have ${\dim \mathcal{D}(\widetilde{X}, \widetilde{D})=n-1}$.
Consider a simplex of a regular CW complex $\mathcal{D}(\widetilde{X}, \widetilde{D})$ of dimension $n-1$. It corresponds to a~point~${p\in \widetilde{X}}$ such that precisely $n$ components of $\widetilde{X}^{=1}$ meet at $p$ (here we use the fact that the pair $(\widetilde{X}, \widetilde{D})$ is log smooth). Note that the point $p$ is fixed by $N$, so its image on~$X$ is a point fixed by~$N$.
By Lemma \ref{lem-faithful-action}, the action of $N$ on the tangent space $T_p \widetilde{X}$ is faithful. Also, for each component $\widetilde{D}_i$ of $\widetilde{D}^{=1}$ the tangent space $T_p \widetilde{D}_i \subset T_p \widetilde{X}$ is $N$-invariant. Therefore, the representation of~$N$ on $T_p \widetilde{X}$ splits into the direct sum of $n$ one-dimensional subrepresentations. We~conclude that~$N$ is an abelian group of rank at most~$n$.
\end{proof}

\begin{lemma}
\label{corollary: coreg-0 criterion}
Let $G$ be a finite group and let $S$ be a $G$-surface. Assume that $\coregG(S)=0$.
%Let $ be the dual complex on which $\coregG(S)=0$ attains.
Then in the exact sequence \eqref{exact-sequence-two-parts-of-G}, the group $N$ is an abelian group of rank at most $2$, and $G_{\mathcal{D}}$ is~either trivial, or cyclic, or a dihedral group. In particular, for any non-abelian simple group acting on~a~surface $S$ one has $\coregG(S)>0$.
\end{lemma}
\begin{proof}
Let $(S, D)$ be a log Calabi--Yau pair such that $\mathrm{coreg}_G(X,D)=0$.
%Let $(S, D)$ be an lc $G$-pair on which $G$-coregularity $0$ attains.
By Proposition \ref{prop-coreg-0-arbitrary-dimension}, in~the~exact sequence \eqref{exact-sequence-two-parts-of-G}, the group $N$ is an abelian group of rank at most $2$.
By Lemma \ref{lemma: dual complex simple}, the dual complex $\mathcal{D}(S, D)$ is homeomorphic either to a circle, or to a segment. In the former case, $G_{\mathcal{D}}$ is~a~subgroup of a dihedral group. In the latter case, $G_{\mathcal{D}}$ is a subgroup of $\CG_2$. None of these groups are simple non-abelian. The claim follows.
\end{proof}

The following observation allows us to study subgroups and quotient groups to show that a~group~$G$ does not satisfy the conditions of Lemma \ref{corollary: coreg-0 criterion}.

\begin{remark}
\label{remark: passing to subgroups}
Assume that a finite group $G$ contains a normal abelian subgroup $N$ of rank at~most~$2$, such that the quotient $H = G / N$ is either trivial, or cyclic, or dihedral. Then for any subgroup $G_1 \subset G$ the group $N_1 = N \cap G_1$ is abelian of rank at most $2$ and normal in $G_1$. Moreover, the quotient $H_1 = G_1 / N_1$ is a subgroup of $H$, and therefore $H_1$ is either trivial, or cyclic, or dihedral.
\iffalse
\begin{equation}
\begin{tikzcd}
1  \arrow{r} & N_1 \arrow[hookrightarrow]{d} \arrow{r} & G_1 \arrow{r}  \arrow[hookrightarrow]{d} & H_1  \arrow[hookrightarrow]{d} \arrow{r} & 1 \\
1  \arrow{r} & N \arrow{r} & G \arrow{r} & H \arrow{r} & 1 \\
\end{tikzcd}
\end{equation}
\fi

If $f \colon G \rightarrow G_2$ is a surjective homomorphism, then the group $N_2 = f(N)$ is abelian of rank at~most~$2$ and normal in $G_2$. Moreover, the quotient $H_2 = G_2 / N_2$ is isomorphic to a quotient of~$H$, and therefore $H_2$ is either trivial, or cyclic, or dihedral.
\end{remark}

Now, we consider $G$-varieties of positive $G$-coregularity.

\begin{lemma}
\label{lemma: coreg 1 fixed}
Let $(X, D)$ be a log Calabi--Yau $G$-pair. %, where $X$ has dimension $n\leqslant 4$.
Assume that $\coregG(X)=n-1$. Then either there exists a $G$-invariant irreducible subvariety on $X$ of dimension at most $n-1$, or there exists a $G$-invariant pair of irreducible subvarieties on $X$ of dimension at most $n-1$.
%Let $(S,D)$ be a log pair, where $S$ is smooth, such that $\coregG\DD(S,D)=1$. Then either $D$ has a $G$-invariant
%component, or $G$ has a fixed point, or $G$ has an invariant set of two points . In the latter case $(S,D)$ is of quadric type.
\end{lemma}
\begin{proof}
By assumption, we have $\dim \mathcal{D}(X, D)=0$.
If $\mathcal{D}(X, D)$ is connected then $\mathcal{D}(X, D)$ is a point. % by Theorem \ref{thm-kollar-xu},
Otherwise $\mathcal{D}(X, D)$ is a pair of points by Proposition \ref{prop-disconnnected-case}.
Let $\varphi\colon (\widetilde{X}, \widetilde{D})\to (X,D)$ be a crepant $G$-equivariant log resolution. Then~$\widetilde{D}^{=1}$ consists of one component or two disjoint components. Moreover, $\widetilde{D}^{=1}$ is obviously $G$-invariant. Therefore the image of~$\widetilde{D}^{=1}$ on $X$ is $G$-invariant and consists of one or two irreducible subvarieties of dimension at most $n-1$.
\end{proof}

\iffalse
\begin{corollary}
Let $S$ be a $G$-surface such that $\coregG(S)=1$. Then either $S$ is $G$-birational to a product $\mathbb{P}^1\times C$, where $C$ is a rational or elliptic curve, or there exists a $G$-invariant curve or a $G$-fixed point on $S$.
\end{corollary}
\fi

%$\mathbb{P}\mathrm{H}^0(S, \OOO(L))$
\begin{lemma}
\label{lemma:fixedcurve}
Let $G$ be a finite group, and let $(S, \Delta)$ be an lc $G$-pair where $S$ is a klt Fano variety.
Put $L=-K_S-\Delta$.
Assume that the linear system $|L|$ is non-empty and base-point-free. Then~${\coregG(S)\leqslant \coregG(S, \Delta)}$. In particular, if $\Delta$ is non-zero and integral then $\mathrm{reg}_G(S)\geqslant 0$.
\end{lemma}
\begin{proof}
\iffalse
Let $\Delta=\sum a_i \Delta_i$ where $a_i\in \mathbb{Q}$ and $0\leqslant a_i\leqslant 1$.
Since $|L|$ is base-point-free, a general element $D\in |L|$ does not pass through singular points of any $\Delta_i$ and $0$-dimensional lc-centers of the pair $(S, \Delta)$.
By Bertini theorem, $D$ is smooth. Put
Consider an element
\[
D' = \frac{1}{|G|} \sum_{g\in G} gD.
\]
Consider a pair $(S, \Delta + D')$. By our choice of $D'$, the pair $(S, \Delta + D')$ is lc. Indeed, this follows from the fact
\fi
Since $L$ is base-point-free, for a general $D\in |L|$
%we have that $-K_S-\Delta$ is nef.
%By \cite[Theorem 1.7]{B19}, there exists an effective $\mathbb{Q}$-divisor $D$ such that
the pair $(S, \Delta + D)$ is a log Calabi--Yau pair. Put
\[
D^G = \frac{1}{|G|} \sum_{g\in G} gD.
\]
Then the pair $(S, \Delta + D^G)$ is a log Calabi--Yau $G$-pair.
Note that the dual complex~${\mathcal{D}(S, \Delta)}$ is a subcomplex of $\mathcal{D}(S, \Delta + D^G)$. Hence
\[
\coregG(S)\leqslant \coregG(S,\Delta+D^G)\leqslant \coregG(S, \Delta).
\]
Note that if $\Delta\neq 0$ and $\Delta$ is integral then $\mathrm{reg}_G(S)\geq \mathrm{reg}_G(S, \Delta)\geqslant 0$,
and the claim follows.
\end{proof}

\iffalse
\begin{proof}
Consider the following diagram
\begin{equation}
\begin{tikzcd}
\mathbb{A}^n \arrow{d}{} \arrow{r}{}  & \mathbb{P}^{n-1} \arrow{d} \\
\mathbb{A}^n/\overline{G} \arrow{r}{} & \mathbb{P}^{n-1}/G
\end{tikzcd}
\end{equation}

\

Assume that the singularity $0\in X=\mathbb{A}^n/\overline{G}$ is exceptional.

\

Assume that $\coregG(\mathbb{P}^{n-1})\leqslant n-1$. Then there exists an effective $G$-invariant divisor $\mathbb{Q}$-divisor $D$ such that the $G$-pair $(\mathbb{P}^{n-1}, D)$ is lc and not klt, we have $K_{\mathbb{P}^{n-1}}+D\sim_{\mathbb{Q}}0$ and $\dim \mathcal{D}(\mathbb{P}^{n-1},D)\geqslant 0$. This implies that there exists a divisor $E$ over $\mathbb{P}^{n-1}$ whose log discrepancy $a(\mathbb{P}^{n-1}, D, E)$ is equal to $0$.

\end{proof}
\fi
\section{Conic bundles}
\label{sec-conic-bundles}
Consider a smooth projective surface $S$ and a $G$-conic bundle structure $f\colon S\to B=\mathbb{P}^1$ on it. We have a~natural exact sequence
\[
1\to G_F\to G\to G_B\to 1,
\]
where the group $G_B$ acts faithfully on $B=\mathbb{P}^1$, and the group $G_F$ acts faithfully on the generic fiber considered as a conic over the function field of the base $\mathbb{K}(B)$. Note that by Lemma \ref{lemma: action on P1} each of~the groups $G_F$ and $G_B$ is isomorphic to one of the following groups: $\CG_n$, $\DG_{2n}$, $\AG_4$, $\SG_4$, $\AG_5$.

\begin{propositionsec}
\label{proposition: CB}
%Assume that $S$ is smooth rational surface of $G$-coregularity $0$. Assume that $S$ admits a structure of a $G$-conic bundle.
Let $f\colon S\to B=\mathbb{P}^1$ be a $G$-conic bundle. If $\mathrm{coreg}_G(S)=0$, then $G_F$ is either cyclic or dihedral, and $G_B$ is either cyclic or dihedral.
\end{propositionsec}
\begin{proof}
Assume that $\mathrm{coreg}_G(S)=0$. It follows that there exists a $G$-invariant $\mathbb{Q}$-divisor $D$ such that~$(S, D)$ is a log Calabi--Yau pair and $\mathrm{coreg}_G(S, D)=0$. Let $(\widetilde{S}, \widetilde{D})$ be a crepant pullback of~$(S, D)$ via a $G$-equivariant log resolution. Let $\widetilde{f}\colon \widetilde{S}\to B=\mathbb{P}^1$ be the induced morphism.
Let~$\widetilde{D}_{ver}$ be the sum of components of $\widetilde{D}$ that lie in the fibers of $\widetilde{f}$. Let $\widetilde{D}_{hor}=\widetilde{D}-\widetilde{D}_{ver}$. Note that each component of $\widetilde{D}_{hor}$ projects onto $B$ surjectively.

Let $F$ be a general fiber of $\widetilde{f}$. Then $F\cong \mathbb{P}^1$. By adjunction, we have
\begin{equation}
\label{eq-adjunction-to-fiber}
0\sim ( K_{\widetilde{S}}+\widetilde{D})|_F = K_F + \widetilde{D}|_F = K_{\mathbb{P}^1} +  \widetilde{D}|_F.
\end{equation}
Note that $\widetilde{D}|_F=\widetilde{D}_{hor}|_F$.
We have $\widetilde{D}|_F\sim 2P$, where $P$ is a point on $F\cong \mathbb{P}^1$. Consider two cases.

Assume that there are no components of $\widetilde{D}_{hor}$ with coefficient $1$, that is, $\widetilde{D}_{hor}^{=1}=0$. Then $\widetilde{D}^{=1}$ is a~union of components of the fibers of $\widetilde{f}$. By Proposition \ref{prop-disconnnected-case}, we see that $\mathcal{D}(\widetilde{S},\widetilde{D})$ is connected, and hence $\widetilde{D}^{=1}$ is connected as well. This means that $\widetilde{D}^{=1}=\sum D_i$ is a chain of smooth rational curves of length at least $2$ that belong to one fiber $F_1$ of $\widetilde{f}$. In particular, $\widetilde{f}(F_1)=Q\in B$ is a~$G$-invariant point. Therefore, $G_B$ is cyclic. Also, there exists a subgroup $H\subset G_F$ of index at most $2$ such that~$H$ preserves each component of $\widetilde{D}^{=1}$. In particular, $H$ has a fixed point $P$ on~$F_1$. By~Lemma~\ref{lem-faithful-action}, the action of $H$ on the tangent space $T_P \widetilde{S}$ is faithful. Since the groups~$\AG_4$, $\SG_4$, $\AG_5$ do not have faithful $2$-dimensional representations, we conclude that $H$ is either cyclic or~dihedral. Also, the groups $\AG_4$, $\SG_4$, $\AG_5$ do not have cyclic or dihedral subgroups of index $2$. Since by construction $H$ has index at~most~$2$ in $G_F$, we conclude that $G_F$ is itself cyclic or dihedral.

Now assume that there is at least one component of $\widetilde{D}_{hor}$ with coefficient $1$. By \eqref{eq-adjunction-to-fiber}, there are at most~$2$ components with coefficient $1$ in $\widetilde{D}_{hor}$. Hence, $G_F$ has an orbit of cardinality at most~$2$ on~$\mathbb{P}^1_{\overline{\mathbb{K}(B)}}$, where $\overline{\mathbb{K}(B)}$ is the algebraic closure of $\mathbb{K}(B)$. Thus, $G_F$ is either cyclic or dihedral.

Assume that $\widetilde{D}_{hor}^{=1}$ has only one component, say $D_1$. Note that $D_1$ is $G$-invariant. Since the dual complex $\mathcal{D}(\widetilde{S},\widetilde{D})$ is either a circle or segment, we have that $D_1$ intersects either one or two components of $\widetilde{D}_{ver}^{=1}$ (see Lemma \ref{lemma: dual complex simple}). It follows that this component or this pair of components is $G$-invariant. Thus, $G_B$ has an orbit of cardinality at most $2$ on $B$, and so $G_B$ is either cyclic or~dihedral.

Finally, assume that $\widetilde{D}_{hor}^{=1}$ has two components, say $D_1$ and $D_2$. If $D_1$ and $D_2$ intersect, the~intersection has cardinality at most $2$ (because $\mathcal{D}(\widetilde{S},\widetilde{D})$ is either a circle or a line segment by~Lemma~\ref{lemma: dual complex simple}). Hence, the image of $D_1\cap D_2$ gives us an orbit of cardinality at most $2$ on $B$, and so $G_B$ is either a cyclic or dihedral. So we are left with the case when $D_1$ and $D_2$ do not intersect. Since $\mathcal{D}(\widetilde{S},\widetilde{D})$ is either circle or a line segment, there exist one or two fibers of $\widetilde{f}$ containing components of $\widetilde{D}_{ver}^{=1}$ that intersect $D_1$ and also containing components of $\widetilde{D}_{ver}^{=1}$ that intersect $D_2$. Then the image of these fibers on $B$ is $G_B$-orbit of cardinality at most $2$, and so $G_B$ is either cyclic or~dihedral. This completes the proof.
%Now assume that $G_F$ is either cyclic or dihedral, and $G_B$ is either cyclic or dihedral.
%Since $H\subset G_F$, and $G_F$ acts on $\mathbb{P}^1_{\mathbb{K}(B)}$ with
%Let $\widetilde{D}_{hor}$ be the sum of components of $\widetilde{D}$ whose restriction to $\widetilde{D}|_F$ is non-zero.
\end{proof}

The next example shows that there exist $G$-conic bundles with positive coregularity, even in the case when $G$ is the trivial group. This shows that the converse to Proposition \ref{proposition: CB} does not hold in~general.

\begin{examplesec}
Consider a  blow up $S$ of $\mathbb{P}^1 \times \mathbb{P}^1$ in $N$ points in a general position.
Then $S$ admits a~structure of a conic bundle $S\to \mathbb{P}^1$ induced by one of the natural projections on $\mathbb{P}^1 \times \mathbb{P}^1$.
One checks that $|-K_S|=|-2K_S|=\varnothing$ for $N\gg0$. Thus, by \cite[Theorem 4]{FFMP22}, one has~${\mathrm{coreg}(S)>0}$.
%We present an example of a conic bundle $f\colon S\to \mathbb{P}^1$ of coregularity $2$.
\end{examplesec}

\section{Del Pezzo surfaces of higher degree}
\label{sec-del-pezzo-surfaces}
In this section we study automorphism groups of smooth del Pezzo surfaces of degree at least $6$ and~compute $G$-coregularity for
various finite groups $G$. Note that automorphisms groups of smooth del Pezzo surfaces over $\mathbb{C}$ are described in \cite{DI09}, and the classification is the same for any algebraically closed field of characteristic~$0$.

\subsection{Automorphisms of $\mathbb{P}^2$}

A smooth del Pezzo surface of degree $9$ is isomorphic to $\mathbb{P}^2$. In this subsection we compute $\coregG(\mathbb{P}^2)$ for any finite group $G\subset \mathrm{Aut}(\mathbb{P}^2)$. Note that finite subgroups of~${\mathrm{PGL}_3(\mathbb{K}) \cong \operatorname{Aut}\left(\mathbb{P}^2\right)}$ are described in Theorem \ref{theorem: PGL3-classification}.

At first we prove the following auxiliary lemma.

\begin{lemma}
\label{lem-B2-positive-coreg}
Let $S$ be a $G$-surface where $G$ is a group isomorphic to a group of type $(\mathrm{B}2)$ (see Remark \ref{remark: case B}). Then $\coregG(S)>0$.
\end{lemma}
\begin{proof}
%Let us show that $\coregG(\mathbb{P}^2) \neq 0$.
Note that a group $G$ of type $(\mathrm{B}2)$ is a subgroup of $\mathrm{GL}_2(\mathbb{K})$, such that $G$ contains a~subgroup~$G_1$ isomorphic to a binary tetrahedral group $\overline{T}$ (see \cite[Subsection 4.2]{DI09}). Observe that a~unique non-trivial normal abelian subgroup in $\overline{T}$ is isomorphic to $\CG_2$, and $\overline{T} / \CG_2 \cong \AG_4$. Therefore by Lemma \ref{corollary: coreg-0 criterion} one has $\operatorname{coreg}_{\overline{T}}(\mathbb{P}^2) > 0$, and by Remark~\ref{rem-monotonicity} one has $\coregG(\mathbb{P}^2) > 0$.
\end{proof}

Now we prove the following proposition.

\begin{proposition}
\label{proposition: dP9}
Let $G$ be a finite group in $\mathrm{PGL}_3(\mathbb{K})=\mathrm{Aut}(\mathbb{P}^2)$.
\begin{enumerate}
\item
If $G$ has type $(\mathrm{A})$, $(\mathrm{B}1)$, $(\mathrm{C})$, or $(\mathrm{D})$ then $\coregG(\mathbb{P}^2)=0$;
\item
if $G$ has type $(\mathrm{B}2)$, $(\mathrm{E})$, or $(\mathrm{H})$ then $\coregG(\mathbb{P}^2)=1$;
\item
if $G$ has type $(\mathrm{F})$, $(\mathrm{G})$, $(\mathrm{I})$, or $(\mathrm{K})$ then $\coregG(\mathbb{P}^2)=2$.
\end{enumerate}
\end{proposition}

\begin{proof}
At first assume that $G$ has type $(\mathrm{A})$, $(\mathrm{B}1)$, $(\mathrm{C})$, or $(\mathrm{D})$. Then there exists a $G$-invariant triple of points not lying on a line. Indeed, in case $(\mathrm{A})$ we can choose three non-collinear fixed points of~$G$, since $G$ is diagonalizable in this case, and in cases $(\mathrm{C})$ or $(\mathrm{D})$ we can choose three non-collinear fixed points of $N$ (see Theorem \ref{theorem: PGL3-classification}) permuted by $G / N$, since $N$ is a diagonalizable abelian group. For case $(\mathrm{B}1)$ we consider the unique fixed point of $G$, and a pair of points on~the~invariant line fixed by the normal cyclic subgroup in $\DG_{2n}$.

Consider three lines $L_1$, $L_2$, and $L_3$ passing through pairs of these points. Then $D = \sum_{i=1}^{3} L_i$ is a $G$-invariant divisor equivalent to $-K_{\mathbb{P}^2}$.
Thus $(\mathbb{P}^2,D)$ is a log Calabi--Yau pair on which zero $G$-coregularity is attained.

Consider case $(\mathrm{B}2)$. %The group $G$ has a fixed point, therefore by Lemma \ref{lemma: fixed point coregularity 1} one has $\coregG(\mathbb{P}^2) \leqslant 1$.
By Lemma \ref{lem-B2-positive-coreg} we have $\coregG(\mathbb{P}^2)>0$. Moreover, there is a unique \mbox{$G$-invariant} line $L$ on $\mathbb{P}^2$, and the linear system $|-K_{\mathbb{P}^2} - L|$ is base point free. Thus~${\coregG(\mathbb{P}^2) \leqslant 1}$ by Lemma \ref{lemma:fixedcurve}. Therefore $\coregG(\mathbb{P}_2) = 1$.

Consider case $(\mathrm{E})$. Assume that $\coregG(\mathbb{P}^2) = 0$. Then the conditions of Lemma \ref{corollary: coreg-0 criterion} should be satisfied, that can be done only if $N \cong \CG_3^2$, and $G_{\mathcal{D}} \cong \CG_4$ in the notation of Lemma \ref{corollary: coreg-0 criterion}. But by Proposition \ref{prop-coreg-0-arbitrary-dimension}, the normal group $N$ must have a fixed point, which is not true (see~\cite[Chapter~V]{Bl17}). We obtain a contradiction. Therefore $\coregG(\mathbb{P}^2) > 0$. Moreover, for cases $(\mathrm{F})$ and~$(\mathrm{G})$ one has $\coregG(\mathbb{P}^2) > 0$, since such groups contain a subgroup isomorphic to a group of~type~$(\mathrm{E})$ (see Remark \ref{rem-monotonicity}).

For cases $(\mathrm{H})$, $(\mathrm{I})$, and $(\mathrm{K})$ the group $G$ is simple. Therefore $\coregG(\mathbb{P}^2) > 0$ by Lemma \ref{corollary: coreg-0 criterion}.

Note that a group of type $(\mathrm{E})$ is a subgroup of a group of type $(\mathrm{G})$. Note that the latter group is the group of automorphisms of the Hesse pencil of cubics
\[
\lambda(x^3 + y^3 + z^3) + \mu xyz = 0
\]
on $\mathbb{P}^2$ (see \cite[Section 4]{AD06}). Moreover, both of these groups contain a normal subgroup $H \cong \CG_3^2 \rtimes \CG_2$ preserving each cubic in the Hesse pencil. If $G$ has type $(\mathrm{E})$ then the quotient $G / H \cong \CG_2$ acts faithfully on the Hesse pencil. A cyclic group acting on a projective line always has a fixed point. Therefore there is a $G$-invariant cubic $C$. Note that the group $G$ is transitive, thus $G$ does not have fixed points, invariant pairs of points and lines. Moreover, $G$ is primitive, hence $C$ is not a union of three lines. Therefore the only possibility for $C$ is to be a smooth cubic curve, and $(\mathbb{P}^2,C)$ is a~log~Calabi--Yau pair on which $G$-coregularity one is attained.

The action of a group of type $(\mathrm{H})$ has the following description. Consider the action of $\AG_5$ on~$\mathbb{P}^1$. Then the anticanonical linear system $|-K_{\mathbb{P}^1}|$ gives a $G$-equivariant embedding $\mathbb{P}^1 \hookrightarrow \mathbb{P}^2$. Thus we have a $G$-invariant smooth conic $C$ on $\mathbb{P}^2$, and the linear system $|-K_{\mathbb{P}^2} - C|$ is base point free. Therefore $\coregG(\mathbb{P}^2) \leqslant 1$ by Lemma \ref{lemma:fixedcurve}. Thus $\coregG(\mathbb{P}^2) = 1$.

Now let us show that for types $(\mathrm{F})$, $(\mathrm{G})$, $(\mathrm{I})$ and $(\mathrm{K})$ one has $\coregG(\mathbb{P}^2)=2$. Assume that~${\coregG(\mathbb{P}^2)=1}$ and $(\mathbb{P}^2,D)$ is a log Calabi--Yau pair on which $G$-coregularity one is attained.

By Lemma \ref{lemma: coreg 1 fixed} there is either a $G$-invariant irreducible subvariety of dimension $0$ or $1$ on $X$, or~a~$G$-invariant pair of such subvarieties. Note that $G$ is intransitive, and therefore $G$ does not have fixed points and invariant lines. Moreover, $G$ does not have invariant pairs of points and pairs of lines, since the line passing through a $G$-invariant pair of fixed points is $G$-invariant, and a point of intersection of a $G$-invariant pair of lines is $G$-fixed.

Therefore $D$ has a $G$-invariant component $C$ with coefficient~$1$, where $C$ is an irreducible conic or a cubic. If $C$ is a conic then $G$ acts faithfully on $C$ since an element of $\mathrm{PGL}_3(\mathbb{K})$ cannot preserve a conic pointwisely. But $G$ cannot act faithfully on a conic by Lemma \ref{lemma: action on P1}.

Now assume that $C$ is an irreducible cubic curve. If $C$ is singular then the singular point \mbox{is $G$-fixed} which is impossible. Therefore $C$ is an elliptic curve, and $G$ should contain a normal abelian subgroup $N$ (corresponding to translations on the elliptic curve) of rank at most $2$, such that the quotient $H = G / N$ is either trivial, or cyclic. For cases $(\mathrm{F})$ and $(\mathrm{G})$ the only non-trivial normal abelian  subgroup is $\CG_3^2$, and for cases $(\mathrm{I})$ and $(\mathrm{K})$ there are no non-trivial normal subgroups in $G$. Therefore $G / N$ is not a cyclic group. The obtained contradiction shows that $\coregG(\mathbb{P}^2) = 2$.
\end{proof}

\begin{remark}
Note that in Proposition \ref{proposition: dP9} we could have used Theorem \ref{thm-ChW13} to characterize groups $G\subset \mathrm{Aut}(\mathbb{P}^2)$ such that $\mathrm{coreg}_G(\mathbb{P}^2)=2$.
\end{remark}

\subsection{Automorphisms of $\mathbb{P}^1 \times \mathbb{P}^1$}

A smooth del Pezzo surface of degree $8$ is isomorphic either to $\mathbb{P}^1 \times \mathbb{P}^1$, or~to~the Hirzebruch surface $\mathbb{F}_1$. In this subsection we consider automorphisms of $\mathbb{P}^1 \times \mathbb{P}^1$. We~recall some results on these automorphisms (for details see \cite[Subsection 4.3]{DI09}).

Put $X \cong C_1 \times C_2$, where $C_i \cong \mathbb{P}^1$, and let $\pi_i\colon X \rightarrow C_i$ be the projections. The group $\operatorname{Pic}\left(X\right) \cong \mathbb{Z}^2$ is generated by two classes of $F_i = \pi_i^{-1}(p_i)$, where $p_i$ is a point on $C_i$. One has $-K_X \sim 2F_1 + 2F_2$. Note that any $g \in \operatorname{Aut}(X)$ acts on $\operatorname{Pic}\left(X\right)$, and preserves $-K_X$ and the intersection form. Therefore either $g$ acts trivially on $\operatorname{Pic}\left(X\right)$, or $g$ permutes $F_1$ and $F_2$.

For any group $G \subset \operatorname{Aut}(X)$ denote by $G_0$ its subgroup consisting of elements, acting trivially on~$\operatorname{Pic}\left(X\right)$. Then $G / G_0$ is either trivial, or is isomorphic to $\CG_2$. Moreover, $G_0$ preserves $F_i$, and each projection $\pi_i \colon X \rightarrow C_i$ is $G_0$-equivariant. These projections induce homomorphisms
$$
f_i \colon G_0 \rightarrow \operatorname{Aut}(C_i) \cong \mathrm{PGL}_2(\mathbb{K}).
$$
Denote the image $f_i(G_0)$ by $A_i$. We can consider $A_i$ as a subgroup of $\operatorname{Aut}_0(X)$ acting faithfully on~$C_i$ and acting trivially on $C_j$, where $i \neq j$. One can easily show that $G_0$ is a subgroup of $A_1 \times A_2$ considered as a subgroup of $\operatorname{Aut}_0(X)$. Moreover, if $A_1$ and $A_2$ are not isomorphic then $G = G_0$. The description of subgroups of direct products $A_1 \times A_2$ is well-known.

We will use the notation and facts introduced above in the whole Section $4.2$ not referring to~them.

We need the~following definition.

\begin{definition}
Let $\alpha_1\colon A_1 \rightarrow R$ and $\alpha_2\colon A_2 \rightarrow R$ be surjective homomorphisms of groups. The~\textit{diagonal product} $A_1 \triangle_R A_2$ of $A_1$ and $A_2$ over their common homomorphic image $R$ is~the~subgroup of $A_1 \times A_2$ of pairs $(a_1, a_2)$ such that $\alpha_1(a_1) = \alpha_2(a_2)$. In other words $A_1 \triangle_R A_2$ is a fibered product of $A_1$ and $A_2$ over $R$.
\end{definition}

By Goursat’s Lemma (cf. \cite[Lemma 4.1]{DI09}) for a subgroup $G_0$ of $A_1 \times A_2 \subset \mathrm{PGL}_2(\mathbb{K}) \times \mathrm{PGL}_2(\mathbb{K})$ such that $f_i(G_0) \cong A_i$, one has $G_0 \cong A_1 \triangle_R A_2$ for a~certain group $R$. Moreover, if $G_0$ is finite then the groups $A_1$ and $A_2$ are finite subgroups of~${\mathrm{PGL}_2(\mathbb{K})}$. Therefore each of the groups $A_1$ and $A_2$ is isomorphic to one of the following groups:~$\CG_n$, $\DG_{2n}$, $\AG_4$, $\SG_4$, or $\AG_5$.

Also we need the following definition.

\begin{definition}
Let $G_0$ be a finite subgroup in $\operatorname{Aut}(C_1 \times C_2)$ acting trivially on $\operatorname{Pic}\left(C_1 \times C_2\right)$, such that $G_0$ acts faithfully on $C_1$ and $C_2$. We will say that the action of $G_0$ on $C_1 \times C_2$ is \emph{diagonal} if~there exists a $G_0$-equivariant isomorphism between $C_1$ and $C_2$.
\end{definition}

\begin{remark}
\label{remark: diagonal-act}
Note that if $G_0$ acts diagonally on $\mathbb{P}^1 \times \mathbb{P}^1$ then $G_0$ is isomorphic to a diagonal product $A \triangle_{A} A$ for some group $A$. But the converse is not true. For example one can consider a~group~${G_0 \cong \CG_5 \triangle_{\CG_5} \CG_5}$ such that a generator of $G_0$ acts on $C_1$ as $(x : y) \mapsto (e^{\frac{2\pi\mathrm{i}}{5}}x : y)$ and on $C_2$ as~${(x : y) \mapsto (e^{\frac{4\pi\mathrm{i}}{5}}x : y)}$.
\end{remark}

\begin{remark}
\label{remark: diagonal-graph}
Also note that if a group $G_0$ acts diagonally on $C_1 \times C_2$, and $\varphi \colon C_1 \rightarrow C_2$ is the corresponding $G_0$-equivariant isomorphism, then the graph of $\varphi$, consisting of points $(x, \varphi(x))$, is~a~$G_0$-invariant subvariety of $C_1 \times C_2$.
\end{remark}

Now we can state the main result of this section.

\begin{proposition}
\label{proposition: dP8}
Let $G$ be a finite subgroup in $\operatorname{Aut}(\mathbb{P}^1 \times \mathbb{P}^1)$, and $G_0$, $A_1$, and $A_2$ are defined as above.
\begin{enumerate}
\item
If each $A_i$ is either cyclic or dihedral group, then $\coregG(\mathbb{P}^1 \times \mathbb{P}^1)=0$.
\item
If one of the groups $A_i$ is either cyclic or dihedral, and the other is not, then
$$
{\coregG(\mathbb{P}^1 \times \mathbb{P}^1)=1}.
$$
\item
If $A_1 \cong A_2 \cong A$, where $A$ is isomorphic to $\AG_4$, $\SG_4$, or $\AG_5$, and $G_0 \cong A \triangle_{A} A$ acts diagonally on $\mathbb{P}^1 \times \mathbb{P}^1$, then $\coregG(\mathbb{P}^1 \times \mathbb{P}^1)=1$.
\item
If each $A_i$ is isomorphic to $\AG_4$, $\SG_4$, or $\AG_5$, and the action of $G_0 \cong A_1 \triangle_R A_2$ on $\mathbb{P}^1 \times \mathbb{P}^1$ is not diagonal (in particular it holds if at least two groups among $A_1$, $A_2$, and $R$ are not~isomorphic) then one has $\coregG(\mathbb{P}^1 \times \mathbb{P}^1)=2$.
\end{enumerate}
\end{proposition}

From Proposition \ref{proposition: dP8} and Proposition \ref{proposition: coregularity=lct} we immediately obtain the following.

\begin{corollary}
\label{corol-quadric-lct>1}
%Let $C_1 \cong C_2 \cong \mathbb{P}^1$, $X \cong C_1 \times C_2$, and let $G\subset \operatorname{Aut}(X)$ be a finite group. Denote by $G_0 \subset G$ a subgroup of elements, acting trivially on $\operatorname{Pic}\left(X\right)$. Then the projections $X \rightarrow C_i$ are $G_0$-equivariant. These projections induce homomorphisms $f_i \colon G_0 \rightarrow \operatorname{Aut}(C_i) \cong \mathrm{PGL}_2(\mathbb{K})$. Denote the image $f_i(G_0)$ by $A_i$.

One has $\mathrm{lct}_G (\mathbb{P}^1 \times \mathbb{P}^1)>1$ if and only if each $A_i$ is isomorphic to $\AG_4$, $\SG_4$, or~$\AG_5$, and the action of $G_0$ on $\mathbb{P}^1 \times \mathbb{P}^1$ is not diagonal (in particular it holds if at least two groups among~$A_1$, $A_2$, and $G_0$ are not isomorphic).
\end{corollary}

For convenience of a reader we decompose the proof of this proposition into several lemmas.

\begin{lemma}
\label{lemma: dP8coreg0}
Let $G$ be a finite subgroup in $\operatorname{Aut}(\mathbb{P}^1 \times \mathbb{P}^1)$, and $G_0$, $A_1$, and $A_2$ are defined as~above. One has $\coregG(\mathbb{P}^1 \times \mathbb{P}^1)=0$ if and only if each $A_i$ is either cyclic, or dihedral group.
\end{lemma}

\begin{proof}
If $\coregG(\mathbb{P}^1 \times \mathbb{P}^1)=0$ then $\coreg_{G_0}(\mathbb{P}^1 \times \mathbb{P}^1)=0$. For each $i$ the projection $\pi_i \colon X \rightarrow C_i$ defines a structure of a $G_0$-equivariant conic bundle on $X$, and the group $A_i$ acts faithfully on the base of this conic bundle. Thus $A_i$ is either cyclic, or dihedral group by Proposition \ref{proposition: CB}.

Now assume that each $A_i$ is either cyclic, or dihedral group. If we find a $G$-invariant divisor
$$
D = F_{11} + F_{12} + F_{21} + F_{22},
$$
where $F_{ij}$ is a fiber of $\pi_i$, then $D$ is equivalent to $-K_X$ and $(X, D)$ is~a~log Calabi--Yau pair on~which zero $G$-coregularity is attained.

Assume that $G = G_0$. Then there exists an $A_i$-invariant pair of points $p_{i1}$ and $p_{i2}$ on $C_i$: if $A_i$ is trivial we can choose any two points, if $A_i$ is cyclic we consider two $A_i$-fixed points, and if $A_i$ is dihedral we consider a pair of points fixed by the normal cyclic subgroup in $\DG_{2n}$. Put $F_{ij} = \pi_i^{-1}(p_{ij})$. Then $D = F_{11} + F_{12} + F_{21} + F_{22}$ is $G$-invariant divisor, such that $(X, D)$ is a log Calabi--Yau pair on which zero $G$-coregularity is attained.

Now assume that $G \neq G_0$. In this case one can see that $A_1 \cong A_2$. If $A_i$ is trivial then $G \cong \CG_2$. The generator $g$ of $G$ permutes the factors of $\mathbb{P}^1 \times \mathbb{P}^1$. Consider any fibers $F_{11}$ and $F_{12}$ of $\pi_1$, and put $F_{2j} = gF_{1j}$. Then $F_{21}$ and $F_{22}$ are fibers of $\pi_2$, and the quadruple $F_{ij}$ is $G$-invariant. Thus~$G$-coregularity zero is attained on the log Calabi--Yau pair $(X, F_{11} + F_{12} + F_{21} + F_{22})$.

If $A_1$ is nontrivial then consider a normal cyclic group $N_1$ in $A_1$. Then $N_1$ is normal in $A_1 \times A_2$. Let~$g$ and $h$ be elements of $G$ not lying in $G_0$. Then one has $gN_1g^{-1} = hN_1h^{-1}$, since $g^{-1}h$ preserves the factors of $\mathbb{P}^1 \times \mathbb{P}^1$ and therefore $g^{-1}h \in G_0 \subset A_1 \times A_2$ and $h^{-1}gN_1g^{-1}h = N_1$. Thus the group~${N_2 = gN_1g^{-1}}$ does not depend on $g$. Moreover, $N_2$ is a normal cyclic group in~$A_2$. For each~$i$ consider a pair of $N_i$-fixed points $p_{i1}$ and $p_{i2}$ on $C_i$. This pair is $A_i$-invariant. Put~${F_{ij} = \pi_i^{-1}(p_{ij})}$. Then one can check that $D = F_{11} + F_{12} + F_{21} + F_{22}$ is a $G$-invariant divisor. Thus $G$-coregularity zero is attained on the log Calabi--Yau pair $(X, D)$.
\end{proof}

\begin{lemma}
\label{lemma: dP8coreg1_1}
Let $G$ be a finite subgroup in $\operatorname{Aut}(\mathbb{P}^1 \times \mathbb{P}^1)$, and $G_0$, $A_1$, and $A_2$ are defined as above. If one of the groups $A_i$ is either cyclic, or dihedral, and the other is not then $\coregG(\mathbb{P}^1 \times \mathbb{P}^1)=1$.
\end{lemma}

\begin{proof}
One has $\coregG(\mathbb{P}^1 \times \mathbb{P}^1) > 0$ by Lemma \ref{lemma: dP8coreg0}. Moreover, $G = G_0$ since $A_1$ and $A_2$ are not isomorphic.

Without loss of generality we may assume that $A_1$ is cyclic or dihedral, and $A_2$ is not. Then there is an $A_1$-invariant pair of points $p_{11}$ and $p_{12}$ on $C_1$, and the pair of fibers $F_{1i} = \pi_1^{-1}(p_{1i})$ \mbox{is $G$-invariant}. The linear system $|-K_{\mathbb{P}^1 \times \mathbb{P}^1} - F_{11} - F_{12}|$ is base point free. Thus~${\coregG(\mathbb{P}^1 \times \mathbb{P}^1) \leqslant 1}$ by Lemma \ref{lemma:fixedcurve}. Therefore $\coregG(\mathbb{P}^1 \times \mathbb{P}^1) = 1$.
\end{proof}

Now it remains to compute $G$-coregularity for the case in which each group $A_i$ is isomorphic to $\AG_4$, $\SG_4$, or $\AG_5$. Note that any such group $G$ contains a subgroup $H = \AG_4 \triangle_{\AG_4} \AG_4 \cong \AG_4$ acting faithfully on both $C_i$. %Indeed, this follows from the fact that the action of $\AG_4$ on $\mathbb{P}^1$ is unique up to conjugacy.

\begin{lemma}
\label{lemma: dP8coreg1_2}
Let $G$ be a finite subgroup in $\operatorname{Aut}(\mathbb{P}^1 \times \mathbb{P}^1)$, and $G_0$, $A_1$, and $A_2$ are defined as~above. Assume that each group $A_i$ is isomorphic to $\AG_4$, $\SG_4$, or $\AG_5$. Then $\coregG(\mathbb{P}^1 \times \mathbb{P}^1)=1$ if~and only if there exists a $G$-invariant irreducible curve on $\mathbb{P}^1 \times \mathbb{P}^1$ in the class $F_1 + F_2$ or $2F_1 + 2F_2$ in~$\operatorname{Pic}\left(\mathbb{P}^1 \times \mathbb{P}^1\right)$.
\end{lemma}

\begin{proof}
Assume that $\coregG(\mathbb{P}^1 \times \mathbb{P}^1)=1$.
By Lemma \ref{lemma: coreg 1 fixed} there is either a $G$-invariant irreducible subvariety of dimension $0$ or $1$ on $X$, or a $G$-invariant pair of such subvarieties. Denote this subvariety or a pair of them by $D'$. Note that any $G$-invariant subvariety is $H$-invariant, where~${H = \AG_4 \triangle_{\AG_4} \AG_4 \cong \AG_4}$ is a subgroup of $G$.

Any $H$-orbit on $C_i \cong \mathbb{P}^1$ consists of~at~least~$4$ points. Therefore any $H$-orbit on $X$ consists of~at~least~$4$ points. Thus $D'$ does not contain isolated points. This means that $D'$ is either an~irreducible $H$-invariant curve, or $H$-invariant pair of irreducible curves.

Note that $D'$ is contained in $D$, where $(\mathbb{P}^1 \times \mathbb{P}^1 ,D)$ is a log Calabi--Yau pair \mbox{on which $G$-coregula}\-rity one is attained. Therefore $D' \sim a_1F_1 + a_2F_2$ in $\operatorname{Pic}\left(X\right)$, where $0 \leqslant a_i \leqslant 2$. Moreover, $a_i > 0$, since an $H$-orbit of a fiber $F_i$ of $\pi_i$ consists of at least $4$ fibers, so $D'$ cannot be contained in~the~fibers of $\pi_i$. Thus if $D'$ is a pair of curves $D'_1$ and $D'_2$ then $D'_1 \sim D'_2 \sim F_1 + F_2$. The~set~${D'_1 \cap D'_2}$ \mbox{is $H$-invariant} and consists of $1$ or $2$ points, since $D'_1 \cdot D'_2 = 2$. That is impossible. Therefore $D'$ is an $H$-invariant curve.

Assume that there exists an $H$-invariant irreducible curve $D_1$ in the class $F_1 + 2F_2$ in $\operatorname{Pic}\left(X\right)$. Then~${\pi_2 \colon D_1 \rightarrow C_2 \cong \mathbb{P}^1}$ is an isomorphism, since $D_1 \cdot F_2 = 1$, and $\pi_1 \colon D_1 \rightarrow C_1$ is a double cover branched at two points, since $D_1 \cdot F_1 = 2$ and $D_1 \cong \mathbb{P}^1$. This pair of points on $C_1$ must \mbox{be $H$-invariant}, which is impossible. Therefore there are no $H$-invariant irreducible curves in the class $F_1 + 2F_2$ in $\operatorname{Pic}\left(X\right)$. Similarly, one can show that there are no $H$-invariant irreducible curves in the class $2F_1 + F_2$ in $\operatorname{Pic}\left(X\right)$.

We obtained that either $D' \sim F_1 + F_2$, or $D' \sim 2F_1 + 2F_2$.

Now assume that there exists a $G$-invariant irreducible curve $D'$ on $\mathbb{P}^1 \times \mathbb{P}^1$ in the class $F_1 + F_2$ or $2F_1 + 2F_2$ in $\operatorname{Pic}\left(\mathbb{P}^1 \times \mathbb{P}^1\right)$. In the former case the linear system $|-K_{\mathbb{P}^1 \times \mathbb{P}^1} - D'|$ is base point free, and $\coregG(\mathbb{P}^1 \times \mathbb{P}^1) = 1$ by Lemma \ref{lemma:fixedcurve}. In the latter case $G$-coregularity one is attained on~a~log Calabi--Yau pair $(\mathbb{P}^1 \times \mathbb{P}^1 , D')$.
\end{proof}

\begin{lemma}
\label{lemma: dP8-diagonal-invariant0}
Let $G$ be a finite subgroup in $\operatorname{Aut}(\mathbb{P}^1 \times \mathbb{P}^1)$, and $G_0$, $A_1$, and $A_2$ are defined as~above. Assume that each group $A_i$ is isomorphic to $\AG_4$, $\SG_4$, or $\AG_5$. If there exists a $G$-invariant irreducible reduced curve $D'$ on $\mathbb{P}^1 \times \mathbb{P}^1$ in the class $2F_1 + 2F_2$ then $G_0 = \AG_4 \triangle_{\AG_4} \AG_4 \cong \AG_4$.
\end{lemma}

\begin{proof}
The curve $D'$ is smooth since otherwise there is a $G$-orbit of cardinality less than $4$ consisting of singular points of $D'$. Thus $D'$ is an elliptic curve.

Note that any element $g \in \operatorname{Aut}(X)$ such that $g \notin \operatorname{Aut}_0(X)$ can fix no more than one point on~a~fiber~$K_1$ of $\pi_1 \colon X \rightarrow C_1$, since the curve $g(K_1)$ is a fiber of $\pi_2 \colon X \rightarrow C_2$, and the only possibility for a $g$-fixed point on $K_1$ is the point $K_1 \cap g(K_1)$. Therefore $g$ cannot pointwisely fix $D'$ since a~general fibre of $\pi_1$ meets $D'$ at two points.

Moreover, if $g \in \operatorname{Aut}_0(X) \subset \operatorname{Aut}(X)$ pointwisely fixes $D'$ then $g$ pointwisely fixes $\pi_i(D') = C_i$, that is possible only for trivial $g$. Therefore $G$ acts faithfully on $D'$, and $G$ should contain a~normal abelian subgroup $N$ corresponding to translations on the elliptic curve. The group $N$ has rank at~most~$2$, and the quotient $G / N$ is either trivial, or~cyclic. Therefore $G_0$ contains a normal abelian subgroup $N_0 = G_0 \cap N$ of rank at most $2$ such the quotient $G_0 / N_0$ is either trivial, or~cyclic. Moreover, for any surjective homomorphism $\pi \colon G_0 \rightarrow G'_0$ the group $N'_0 = \pi(N_0)$ is a normal abelian subgroup of $G'_0$ of rank at most $2$ and the quotient $G'_0 / N'_0$ is either trivial, or cyclic.

Note that the homomorphisms $G_0 \rightarrow A_i$ are surjective, and there are no normal abelian subgroups in $\SG_4$ and $\AG_5$ such that the quotient is cyclic. Therefore any $A_i$ cannot be isomorphic to $\SG_4$ or~$\AG_5$, and $G_0 \cong \AG_4 \triangle_R \AG_4$, where $R \cong \AG_4$, $R \cong \CG_3$, or $R = \{1\}$. Moreover, the groups $\AG_4 \triangle_{\CG_3} \AG_4$ \mbox{and $\AG_4 \times \AG_4$} do not contain normal abelian subgroups of rank $2$ such that the quotient is cyclic, since the commutator subgroups of these groups are isomorphic to $\CG_2^4$. Thus $G_0 = \AG_4 \triangle_{\AG_4} \AG_4 \cong \AG_4$.
\end{proof}

\begin{lemma}
\label{lemma: dP8-diagonal-invariant}
The group $H = \AG_4 \triangle_{\AG_4} \AG_4 \cong \AG_4$ acting on $X = \mathbb{P}^1 \times \mathbb{P}^1$ has a unique invariant curve $D'$ in the class $F_1 + F_2$ in $\operatorname{Pic}\left(X\right)$. Moreover, if for a group $G \subset \operatorname{Aut}(X)$ one has $G_0 = H$, then $D'$ is $G$-invariant.
\end{lemma}

\begin{proof}
Any curve in the class $F_1 + F_2$ can be obtained as an intersection of $X \subset \mathbb{P}^3$ and a plane in~$\mathbb{P}^3$.

Any faithful action of $\AG_4$ on $\mathbb{P}^1$ corresponds to a faithful action of the binary tetrahedral group~$\overline{\AG}_4$ on the two-dimensional vector space over $\mathbb{K}$. Let $\rho_i \colon \overline{\AG}_4 \rightarrow \operatorname{GL}(V_i)$ be the two-dimensional representation corresponding to the action of $H \cong \AG_4$ on $C_i=\mathbb{P}(V_i)$. The embedding $X \cong C_1 \times C_2 \hookrightarrow \mathbb{P}^3$ corresponds to the four-dimensional representation $\rho_1 \otimes \rho_2 \colon \overline{\AG}_4 \rightarrow \operatorname{GL}(V_1 \otimes V_2)$. This representation is not faithful, since the image of the center of $\overline{\AG}_4$ for any faithful two-dimensional representation is generated by a scalar matrix corresponding to the multiplication by $-1$. Therefore~${\operatorname{Ker}(\rho_1 \otimes \rho_2) = Z}$ is the center of $\overline{\AG}_4$, and we obtain a faithful four-dimensional representation~${\rho \colon H \rightarrow \operatorname{GL}(V)}$, where~${H \cong \AG_4 \cong \overline{\AG}_4 / Z}$ and $V = V_1 \otimes V_2$.

Any faithful four-dimensional representation of $\AG_4$ can be decomposed into a direct sum of a~one-dimensional representation of $\AG_4$ and the irreducible three-dimensional representation of $\AG_4$.
The representation $\rho$ induces a representation $\rho^*\colon H \rightarrow \operatorname{GL}(V^*)$, which is also a direct sum of the irreducible three-dimensional representation of $\AG_4$ and a one-dimensional representation. Invariant hyperplanes of $H$ in $\mathbb{P}^3$ correspond to $H$-invariant one-dimensional subspaces of $V^*$. There is exactly one such subspace in $V^*$, and therefore there is exactly one $H$-invariant hyperplane in $\mathbb{P}^3$ and $H$-invariant curve $D'$ in the class $F_1 + F_2$ in $\operatorname{Pic}\left(X\right)$.

If for a group $G \subset \operatorname{Aut}(\mathbb{P}^1 \times \mathbb{P}^1)$ one has $G_0 = H$, then the group $G / G_0 \cong \CG_2$ acts on the set of~$G_0$-invariant curves in the class $F_1 + F_2$, but this set consists of a unique curve $D'$. Therefore~$D'$ is $G$-invariant.
\end{proof}

\begin{corollary}
\label{corollary: dP8-invariant-curves}
Let $G$ be a finite subgroup in $\operatorname{Aut}(\mathbb{P}^1 \times \mathbb{P}^1)$, and $G_0$, $A_1$, and $A_2$ are defined as~above. Assume that each group $A_i$ is isomorphic to $\AG_4$, $\SG_4$, or $\AG_5$. If there exists a $G$-invariant irreducible curve on $\mathbb{P}^1 \times \mathbb{P}^1$ in the class $2F_1 + 2F_2$ then there exists a $G$-invariant irreducible curve in the class $F_1 + F_2$.
\end{corollary}

\begin{proof}
If there exists a $G$-invariant irreducible curve on $\mathbb{P}^1 \times \mathbb{P}^1$ in the class $2F_1 + 2F_2$ then
$$
{G_0 = \AG_4 \triangle_{\AG_4} \AG_4 \cong \AG_4}
$$
by Lemma \ref{lemma: dP8-diagonal-invariant0}. Thus there exists a unique $G$-invariant curve in the class $F_1 + F_2$ by~Lemma~\ref{lemma: dP8-diagonal-invariant}.
\end{proof}

Now we want to find possibilities for finite subgroups $G \subset \operatorname{Aut}(C_1 \times C_2)$ preserving an irreducible curve $D'$ in the class $F_1 + F_2$. At first note that the subgroup $G_0 \subset G$ acts diagonally \mbox{on $X \cong C_1 \times C_2$}, since the projections $\pi_i \colon X \rightarrow C_i$ are $G_0$-equivariant, and therefore these projections induce \mbox{$G_0$-equivariant} isomorphism between $D'$ and $C_i$. Thus there exists a $G_0$-equivariant isomorphism between $C_1$ and $C_2$.

\begin{lemma}
\label{lemma: dP8-diagonal-stabilizer}
Let $D'$ be an irreducible curve in the class $F_1 + F_2$ on $\mathbb{P}^1 \times \mathbb{P}^1$. Then the subgroup~$G'$ of $\operatorname{Aut}(\mathbb{P}^1 \times \mathbb{P}^1)$ preserving $D'$ is isomorphic to $\operatorname{PGL}_2(\mathbb{K}) \times \CG_2$.
\end{lemma}

\begin{proof}
Note that the group $\operatorname{Aut}(\mathbb{P}^1 \times \mathbb{P}^1)$ acts transitively on the set of irreducible curves in the class~${F_1 + F_2}$. Therefore we may assume that $D'$ consists of points of type $((x : y), (x : y))$ in~$\mathbb{P}^1 \times \mathbb{P}^1$.

The group $G'$ contains a normal subgroup $N'$ that preserves $D'$ pointwise. If $g_0 \in N'$ preserves the projections $\pi_i$, then $g_0$ acts trivially on $C_i$. Thus $g_0$ is trivial, and $N'$ is either trivial or is isomorphic to $\CG_2$ permuting the factors of $\mathbb{P}^1 \times \mathbb{P}^1$. It is easy to see that the element
$$
g \colon ((x_1 : y_1), (x_2 : y_2)) \mapsto ((x_2 : y_2), (x_1 : y_1))
$$
belongs to $N'$. Thus $N' \cong \CG_2$.

The quotient $G' / N'$ acts faithfully on $D'$, and therefore $G' / N'$ is a subgroup of $\operatorname{PGL}_2(\mathbb{K})$. Consider an action of the group $\operatorname{PGL}_2(\mathbb{K})$ on $\mathbb{P}^1 \times \mathbb{P}^1$ such that the action on the both factors is the same. This action preserves $D'$, and any element of $\operatorname{PGL}_2(\mathbb{K})$ acts faithfully on $D'$ and commutes with $N'$. Thus $G' \cong \operatorname{PGL}_2(\mathbb{K}) \times \CG_2$.
\end{proof}

\begin{corollary}
\label{corollary: dP8-class}
Let $G$ be a finite subgroup in $\operatorname{Aut}(\mathbb{P}^1 \times \mathbb{P}^1)$, and let $G_0$, $A_1$, and $A_2$ be defined as above. Assume that each group $A_i$ is isomorphic to $\AG_4$, $\SG_4$, or $\AG_5$. Then $\coregG(\mathbb{P}^1 \times \mathbb{P}^1)=1$ if and only if $A_1 \cong A_2 \cong A$, $G_0 \cong A \triangle_{A} A$, and $G_0$ acts diagonally on $\mathbb{P}^1 \times \mathbb{P}^1$.
\end{corollary}

\begin{proof}
By Lemma \ref{lemma: dP8coreg1_2} one has $\coregG(\mathbb{P}^1 \times \mathbb{P}^1)=1$ if and only if there exists a $G$-invariant irreducible curve $D'$ on $\mathbb{P}^1 \times \mathbb{P}^1$ in the class $F_1 + F_2$ or $2F_1 + 2F_2$ in $\operatorname{Pic}\left(\mathbb{P}^1 \times \mathbb{P}^1\right)$. Moreover, by~Corollary \ref{corollary: dP8-invariant-curves} we can assume that $D'$ has the class $F_1 + F_2$.

Assume that there exists a $G$-invariant irreducible curve $D'$ in the class $F_1 + F_2$. Then $G$ is a~subgroup of $\operatorname{PGL}_2(\mathbb{K}) \times \CG_2$, where $\operatorname{PGL}_2(\mathbb{K})$ acts diagonally on $\mathbb{P}^1 \times \mathbb{P}^1$, by Lemma \ref{lemma: dP8-diagonal-stabilizer}. Therefore the group $G_0$ acts diagonally on $\mathbb{P}^1 \times \mathbb{P}^1$ and by Remark \ref{remark: diagonal-act} one has $A_1 \cong A_2 \cong A$, $G_0 \cong A \triangle_{A} A$.

Now assume that $A_1 \cong A_2 \cong A$, $G_0 \cong A \triangle_{A} A$ and $G_0$ acts diagonally on $\mathbb{P}^1 \times \mathbb{P}^1$. Then by~Remark~\ref{remark: diagonal-graph} there exists $G_0$-invariant irreducible curve $D'$ on $\mathbb{P}^1 \times \mathbb{P}^1$ in the class $F_1 + F_2$. The subgroup $G_0$ is normal in $G$, therefore for any $g \in G$ the curve $gD'$ is $G_0$-invariant, since any element of $G_0$ can be written as $gg_0g^{-1}$, where $g_0 \in G_0$. But by Lemma \ref{lemma: dP8-diagonal-invariant} there is unique $G_0$-invariant curve in the class $F_1 + F_2$. Thus $D'$ is $G$-invariant.
\end{proof}

\begin{remark}
Note that we can describe groups appearing in Corollary \ref{corollary: dP8-class} more explicitly. By Lemma \ref{lemma: dP8-diagonal-stabilizer} if $\coregG(\mathbb{P}^1 \times \mathbb{P}^1)=1$, then $G$ is a finite subgroup of $\operatorname{PGL}_2(\mathbb{K}) \times \CG_2$, and the image of the projection $G \rightarrow \operatorname{PGL}_2(\mathbb{K})$ is $\AG_4$, $\SG_4$, or $\AG_5$. Let us recall that any subgroup of~a~direct product is a diagonal product $B_1 \triangle_{R} B_2$ for some groups $B_1$, $B_2$, and $R$. In our case the group~$B_1$ is $\AG_4$, $\SG_4$, or $\AG_5$, the group $B_2$ is either trivial, or is isomorphic to~$\CG_2$, and $R$ is either trivial, or is isomorphic to~$\CG_2$ (that is possible only for $B_1 \cong \SG_4$). Thus there are seven possibilities for $G$:
\begin{itemize}
\item $G = G_0 \cong \AG_4 \triangle_{\AG_4} \AG_4 \cong \AG_4$;
\item $G \cong (\AG_4 \triangle_{\AG_4} \AG_4) \times \CG_2 \cong \AG_4 \times \CG_2$;
\item $G = G_0 \cong \SG_4 \triangle_{\SG_4} \SG_4 \cong \SG_4$;
\item $G \cong (\SG_4 \triangle_{\SG_4} \SG_4) \triangle_{\CG_2} \CG_2 \cong \SG_4$;
\item $G \cong (\SG_4 \triangle_{\SG_4} \SG_4) \times \CG_2 \cong \SG_4 \times \CG_2$;
\item $G = G_0 \cong \AG_5 \triangle_{\AG_5} \AG_5 \cong \AG_5$;
\item $G \cong (\AG_5 \triangle_{\AG_5} \AG_5) \times \CG_2 \cong \AG_5 \times \CG_2$.
\end{itemize}
\end{remark}

\begin{remark}
Note that in Corollary \ref{corollary: dP8-class} conditions $A_1 \cong A_2 \cong A$ and $G_0 \cong A \triangle_{A} A$ do not imply that the action of $G_0$ is diagonal. One can consider two non-isomorphic three-dimensional representations of $\AG_5$, and two corresponding two-dimensional representations $\rho_i \colon \overline{\AG}_5 \rightarrow \operatorname{GL}(V_i)$ of~$\overline{\AG}_5$ obtained as the preimages of the double cover $\operatorname{SU}_2(\mathbb{C}) \rightarrow \operatorname{SO}_3(\mathbb{R})$. Consider the action of $\AG_5$ on $C_1$ and $C_2$ given by these representations.

The embedding $X \cong C_1 \times C_2 \hookrightarrow \mathbb{P}^3$ corresponds to the four-dimensional representation
$$
\rho_1 \otimes \rho_2 \colon \overline{\AG}_5 \rightarrow \operatorname{GL}(V_1 \otimes V_2).
$$
This representation is not faithful, since the image of the center of $\overline{\AG}_5$ for any faithful two-dimensional representation is generated by a scalar matrix corresponding to the multiplication by $-1$. Therefore~${\operatorname{Ker}(\rho_1 \otimes \rho_2)}$ is the center of $\overline{\AG}_5$, and we obtain a faithful four-dimensional representation ${\rho \colon \AG_5 \rightarrow \operatorname{GL}(V)}$, where $V = V_1 \otimes V_2$. But this representation is irreducible, and therefore there are no $\AG_5$-invariant curves in the class $F_1 + F_2$ in $\operatorname{Pic}\left(X\right)$.
\end{remark}

%There are three irreducible one-dimensional representations of $\AG_4$, but each faithful four-dimensional representation of $\AG_4$ can be obtained from other by tensoring on a one-dimensional representation of $\AG_4$ and conjugation. Therefore all faithful four-dimensional representations of $\AG_4$ give the same action on $\mathbb{P}^3$, and we can assume that the representation $\rho \colon H \rightarrow \operatorname{GL}(V)$ is a direct sum of the irreducible three-dimensional representation of $\AG_4$ and the trivial one-dimensional representation.

%To describe $H$-invariant curves with class $2F_1 + 2F_2$ in $\operatorname{Pic}\left(X\right)$ consider the representation $\mathrm{S}^2\rho^* \colon H \rightarrow \operatorname{GL}(\mathrm{S}^2V^*)$. This representation is a direct sum of two trivial one-dimensional representations, two different non-trivial one-dimensional representations and two irreducible three-dimensional representations. Note that each $H$-invariant curve with class $2F_1 + 2F_2$ corresponds to $H$-invariant two-dimensional subspase $U$ of $\mathrm{S}^2V^*$ containing $H$-invariant one-dimensional subspace $W$ corresponding to $X$. It is easy to check that the action of $H$ on $W$ is trivial. Therefore there are exactly three possibilities for $U$: either $H$ trivially acts on $U$, or $U = W \oplus W_i$, where $W_i$ is a one-dimensional subspace corresponding to a non-trivial one-dimensional representation of $H$. Thus there are exactly three $H$-invariant curves with class $2F_1 + 2F_2$ in $\operatorname{Pic}\left(X\right)$.

Now we can collect the obtained results and prove Proposition \ref{proposition: dP8}.

\begin{proof}[Proof of Proposition \ref{proposition: dP8}]
By Lemma \ref{lemma: dP8coreg0} one has~${\coregG(\mathbb{P}^1 \times \mathbb{P}^1)=0}$ if and only if each $A_i$ is either cyclic, or dihedral group. Moreover, by Lemma \ref{lemma: dP8coreg1_1} if one of the groups $A_i$ is either cyclic, or dihedral, and the other is not, then $\coregG(\mathbb{P}^1 \times \mathbb{P}^1)=1$.

Therefore we can assume that each $A_i$ is isomorphic to $\AG_4$, $\SG_4$, or $\AG_5$. By Corollary \ref{corollary: dP8-class} one has~${\coregG(\mathbb{P}^1 \times \mathbb{P}^1)=1}$ if and only if $A_1 \cong A_2 \cong A$, $G_0 \cong A \triangle_{A} A$ and $G_0$ acts diagonally on~$\mathbb{P}^1 \times \mathbb{P}^1$. Otherwise one has $\coregG(\mathbb{P}^1 \times \mathbb{P}^1)=2$, and the action of $G_0$ is not diagonal (in~particular it holds if at least two groups among $A_1$, $A_2$, and $R$ are not isomorphic).
\end{proof}

\subsection{Automorphisms of $\mathbb{F}_1$}

The surface $\mathbb{F}_1$ is given by the blowup $f\colon \mathbb{F}_1 \rightarrow \mathbb{P}^2$ of $\mathbb{P}^2$ at~a~point~$p$. Let $E$ be the exceptional divisor of this blowup. Note that $E$ is a unique $(-1)$-curve on $\mathbb{F}_1$. Therefore for any group $G$ acting on $\mathbb{F}_1$ the morphism $f$ is $G$-equivariant, and $G$ acts on $\mathbb{P}^2$ with a fixed point~$p$.
We can apply Theorem \ref{theorem: PGL3-classification} to classify groups acting on $\mathbb{F}_1$ as follows: the groups having fixed points correspond to cases $(\mathrm{A})$ and $(\mathrm{B})$. Also we will use Remark \ref{remark: case B} to divide case $(\mathrm{B})$ into two subcases $(\mathrm{B}1)$ and $(\mathrm{B}2)$.

\begin{proposition}
\label{proposition: F1}
Let $G$ be a finite group acting on $\mathbb{F}_1$. If $G$ has type $(\mathrm{A})$ or $(\mathrm{B}1)$, then one has~$\coregG(\mathbb{F}_1)=0$; if $G$ has type  $(\mathrm{B}2)$, then one has $\coregG(\mathbb{F}_1)=1$.
\end{proposition}

\begin{proof}
At first assume that $G$ has type $(\mathrm{A})$ or $(\mathrm{B}1)$. Consider the action of $G$ on $\mathbb{P}^2$. Then there exists a $G$-invariant triple of points not lying on a line. In case $(\mathrm{A})$ the triple consists of $p$ and two other fixed points of $G$, we may assume that these three points are non-collinear, since $G$ is diagonalizable in this case. In case $(\mathrm{B}1)$ the triple is $p$ and a pair of points on the invariant line fixed by the normal cyclic subgroup in $\DG_{2n}$.

Consider three lines $L_1$, $L_2$, and $L_3$ passing through pairs of these points. Then
$$
{D = E + \sum_{i=1}^{3} f^{-1}_*(L_i)}
$$
is a $G$-invariant divisor equivalent to $-K_{\mathbb{F}_1}$.
Thus $(\mathbb{F}_1,D)$ is a log Calabi--Yau pair on which \mbox{$G$-coregularity} zero is attained.

Consider case $(\mathrm{B}2)$. By Lemma \ref{lem-B2-positive-coreg} one has $\coregG(\mathbb{F}_1) > 0$.
The linear system $|-K_{\mathbb{F}_1} - E|$ is base point free. Therefore $\coregG(\mathbb{F}_1) \leqslant 1$ by Lemma \ref{lemma:fixedcurve}. Thus $\coregG(\mathbb{F}_1) = 1$.
\end{proof}

\subsection{Automorphisms of del Pezzo surface of degree $7$}
\label{subsect-dp7}
%We will prove the following proposition.
\begin{proposition}
\label{proposition: dP7}
If $G$ is a finite group acting on a smooth del Pezzo surface $S_7$ of degree $7$ then one has~$\coregG(S_7)=0$.
\end{proposition}
\begin{proof}
A smooth del Pezzo surface $S_7$ of degree $7$ is given by the blowup $f\colon S_7 \rightarrow \mathbb{P}^2$ of $\mathbb{P}^2$ at two distinct points $p_1$ and $p_2$. Let $E_1$ and~$E_2$ be the corresponding exceptional divisors of this blowup. We denote by $E$ the proper transform of~the line passing through $p_1$ and $p_2$. Note that there are exactly three $(-1)$-curves on $S_7$, and the pair~$E_1$ and~$E_2$ is a unique pair of disjoint $(-1)$-curves. Therefore for any group $G$ acting on $S_7$ the morphism $f$ is $G$-equivariant, and the pair of points $p_1$ and $p_2$ is $G$-invariant.

The line $f(E)$ passing through the points $p_1$ and $p_2$ is $G$-invariant, which is possible only if $G$ has type $(\mathrm{A})$ or $(\mathrm{B})$ of Theorem \ref{theorem: PGL3-classification}. For both these cases there exists a $G$-fixed point $p$, not lying on $f(E)$. Let $L_1$ and $L_2$ be the lines passing through $p$ and $p_1$ or $p_2$ respectively. Then this pair of lines is $G$-invariant. The $G$-invariant divisor $D = E_1 + E + E_2 + \sum_{i=1}^{2} f^{-1}_*(L_i)$ is equivalent to~$-K_{S_7}$. Thus $(S_7,D)$ is a log Calabi--Yau pair on which $G$-coregularity zero is attained.
\end{proof}

\subsection{Automorphisms of del Pezzo surface of degree $6$}
%We have proved the following proposition.
\begin{proposition}
\label{proposition: dP6}
If $G$ is a finite group acting on a smooth del Pezzo surface $S_6$ of degree $6$ then one has~$\coregG(S_6)=0$.
\end{proposition}
\begin{proof}
A smooth del Pezzo surface $S_6$ of degree $6$ is given by the blowup of $\mathbb{P}^2$ at three distinct points not lying on a line. Note that there are exactly six $(-1)$-curves on $S_6$: the preimages of these points and the proper transforms of the lines passing through pairs of these points. Consider a divisor $D$ which is a sum of these six $(-1)$-curves. For any group $G$ acting on $S_6$ the divisor $D$ is obviously $G$-invariant. Moreover, $D$ is equivalent to~$-K_{S_6}$. Thus $(S_6,D)$ is a log Calabi--Yau pair on which $G$-coregularity zero is attained.
\end{proof}

\subsection{General results}
In this subsection we summarize obtained results on $G$-coregularity for different finite groups $G$ acting on smooth del Pezzo surfaces of degree at least $6$, and formulate some general statements based on these results.

Note that any smooth del Pezzo surface of degree $6$ or greater is a toric variety. The next proposition allows one to find $G$-coregularity for many groups $G$ acting on a toric variety of arbitrary dimension.

\begin{proposition}
\label{prop-toric-varieities}
Let $X$ be a projective toric variety.
Let $G$ be an algebraic group
that normalizes a maximal torus $T$ in $\mathrm{Aut}(X)$.
Then $\mathrm{coreg}_G(X)=0$.
\end{proposition}
\begin{proof}
Let $D=\sum D_i\sim -K_X$ be the sum of $T$-invariant divisors. It follows that $D$ is $G$-invariant.
It is well known that $\mathrm{Aut}(X)$ is a linear algebraic group, and the normalizer $N_T$ of the torus $T$ is its algebraic subgroup.
Let $(\widetilde{X}, \widetilde{D})$ be a $N_T$-equivariant log resolution of $(X, D)$, so in particular~$(\widetilde{X}, \widetilde{D})$ is a log smooth $G$-pair which is $G$-crepant birational to a log CY pair $(X, D)$. Note that the pair~$(\widetilde{X}, \widetilde{D})$ is toric, and
%Note that complement to the open orbit of the torus $T$ is a union of torus-invariant divisors $D=\sum D_i$.
%It is well known that the pair $(X, D)$ is lc,
$\widetilde{D}$ admits a zero-dimensional stratum. It follows that
$$
\dim \mathcal{D}(\widetilde{X}, \widetilde{D})=\dim \mathcal{D}(X, D)=\dim X-1,
$$
and so $\mathrm{coreg}_G(X)=0$.
\end{proof}

\begin{corollary}
\label{cor-coreg-0-dp6}
Let $G$ be a finite group which acts on a smooth del Pezzo surface $X$ of degree $d\geqslant 6$. Then $\coregG(X)=0$ if and only if $G$ normalizes a torus $T \cong (\mathbb{K}^*)^2$ in $\mathrm{Aut}(X)$.
\end{corollary}
\begin{proof}
By Proposition \ref{prop-toric-varieities}, if $G$ normalizes a torus $T$ in $\mathrm{Aut}(X)$ then $\mathrm{coreg}_G(X)=0$. Assume that $\mathrm{coreg}_G(X)=0$ for some finite group $G\subset \mathrm{Aut}(X)$.

If $X$ has degree $d=6$, then any subgroup of $\mathrm{Aut}(X)$ preserves a torus $(\mathbb{K}^*)^2$ which is the open orbit of a $(\mathbb{K}^*)^2$-action on $X$. Indeed, this torus is the complement to the union of six $(-1)$-curves whose sum is equivalent to $-K_X$. Hence the claim is clear in this case.

%By Proposition \ref{proposition: dP6}, the coregularity of any finite subgroup of $\mathrm{Aut}(X)$ is equal to $0$, hence the claim is clear in this case.

Assume that $X$ has degree $d=7$. %then any finite subgroup of $\mathrm{Aut}(X)$ has coregularity $0$ by Proposition \ref{proposition: dP7}.
As shown in Section \ref{subsect-dp7}, any finite subgroup $G$ of $\mathrm{Aut}(X)$ descends to a~subgroup of $\mathrm{Aut}(\mathbb{P}^2)$ via the map $f\colon X \rightarrow \mathbb{P}^2$, and preserves a triangle of lines on $\mathbb{P}^2$. Since the morphism
 is toric, it follows that $G$ normalizes a torus in $\mathrm{Aut}(X)$.

Assume that $X$ has degree $8$ and $X \cong \mathbb{F}_1$. Then any finite subgroup $G$ of $\mathrm{Aut}(X)$ descends to~a~subgroup of $\mathrm{Aut}(\mathbb{P}^2)$ via the map $f\colon \mathbb{F}_1 \rightarrow \mathbb{P}^2$. From  Proposition \ref{proposition: F1} it follows that the group~$G$ has type $(\mathrm{A})$ or $(\mathrm{B}1)$. In both cases, $G$ preserves a triangle of lines on $\mathbb{P}^2$. Since the morphism  is toric, it follows that $G$ normalizes a torus in $\mathrm{Aut}(X)$.

Assume that $X$ has degree $8$ and $X \cong \mathbb{P}^1\times \mathbb{P}^1$. From the proof of Lemma \ref{lemma: dP8coreg0} it follows that there exists a $G$-invariant divisor $D = F_{11} + F_{12} + F_{21} + F_{22}$, where $F_{ij}$ is a fiber of the projection of $X$ to the $i$-th factor. Moreover, the complement $X \setminus D$ is a torus, and $G$ preserves this torus. Therefore $G$ normalizes a torus in $\mathrm{Aut}(X)$.

Now assume that $X$ has degree $9$, so $X \cong \mathbb{P}^2$.
By Proposition \ref{proposition: dP9}, if $G$ has coregularity $0$ then~$G$ has type $(\mathrm{A})$, $(\mathrm{B}1)$, $(\mathrm{C})$, or $(\mathrm{D})$.
Clearly, in the types $(\mathrm{A})$, $(\mathrm{C})$, or $(\mathrm{D})$ the group $G$ belongs to a subgroup
$(\mathbb{K}^*)^2\rtimes \SG_3$
of $\mathrm{PGL}_2(\mathbb{K})$, so $G$ normalizes a torus. In the case $(\mathrm{B}1)$, the group $G$ preserves the union of three lines, so $G$ preserves its complement.
\end{proof}

Now we can prove the following result.
\begin{theorem}
\label{main-thm-del-pezzo-sec}
Assume that $S$ is a smooth del Pezzo surface of degree $d\geqslant 6$, and $G$ is a finite subgroup in $\operatorname{Aut}(S)$. Then
\begin{enumerate}
\item
$\mathrm{coreg}_G(S)=0$ if and only if $G$ belongs to the normalizer of a torus $(\mathbb{K}^*)^2$ in $\mathrm{Aut}(S)$;
\item
if $\mathrm{coreg}_G(S)>0$ then
$\mathrm{coreg}_G(S)=1$ if and only if there exists a $G$-invariant curve $C$ on $S$ such that $-K_S-C$ is effective;
\item
$\mathrm{coreg}_G(S)=2$ otherwise.
\end{enumerate}
\end{theorem}
\begin{proof}
The first claim follows from Corollary \ref{cor-coreg-0-dp6}.

%Assume that there exists a $G$-invariant curve $C$ on $S$ such that $-K_S-C$ is effective. Let $D\sim -K_S-C$ be an effective divisor. By averaging

Assume that $\mathrm{coreg}_G(S)=1$.
If $S\cong\mathbb{P}^2$ then by Proposition \ref{proposition: dP9} the group $G$ has type $(\mathrm{B}2)$, $(\mathrm{E})$, or $(\mathrm{H})$. For these cases there exists a $G$-invariant line, cubic, or conic curve on $\mathbb{P}^2$, respectively. If~${S\cong\mathbb{P}^1\times\mathbb{P}^1}$, then from Proposition \ref{proposition: dP8} it follows that $G$ has on $\mathbb{P}^1\times \mathbb{P}^1$ either a $G$-invariant pair of fibers of a projection $\mathbb{P}^1\times \mathbb{P}^1\to\mathbb{P}^1$, or a $G$-invariant irreducible curve in the class
$$
-\frac{1}{2}K_{\mathbb{P}^1\times\mathbb{P}^1}\in \mathrm{Pic}(\mathbb{P}^1\times\mathbb{P}^1).
$$
If $S\cong \mathbb{F}_1$ then the unique $(-1)$-curve is $G$-invariant. If $S$ is a smooth del Pezzo surface of degree $7$ or $6$ then $\mathrm{coreg}_G(S)=1$ is not possible by Proposition \ref{proposition: dP7} and Proposition \ref{proposition: dP6}.

Assume that $\mathrm{coreg}_G(S)=2$. %Then $S$ is isomorphic either to $\mathbb{P}^2$ or to $\mathbb{P}^1\times\mathbb{P}^1$ by Proposition \ref{proposition: F1}, Proposition \ref{proposition: dP7} and Proposition \ref{proposition: dP6}.
%If $S\cong\mathbb{P}^2$ then by Proposition \ref{proposition: dP9} the group $G$ has type $(\mathrm{F})$, $(\mathrm{G})$, $(\mathrm{I})$, or $(\mathrm{K})$.
By Proposition \ref{proposition: coregularity=lct}, we have $\mathrm{lct}_G(\mathbb{P}^2)>1$. This implies that there is no $G$-invariant curve $C$ on $S$ such that $-K_S-C$ is effective.
\end{proof}

\section{$G$-coregularity with linear systems}
\label{sec-coreg-linear-systems}
In this section, we extend the theory of coregularity to the setting of linear systems, which allows us to give a more general definition of $G$-coregularity which could be useful for infinite groups.
We refer to \cite{FS23} for a relevant theory of dual complexes for generalized pairs.

Let $G$ be an algebraic %\footnote{We may consider an even more general setting and not assume that $G$ is algebraic. However, in what follows, we need the existence of $G$-equivariant resolution, which exists for varieties endowed with the action of an algebraic group.}
group that acts on a normal projective $\mathbb{Q}$-factorial variety $X$.

\begin{definitionsec}
Following \cite{Al94}, by a \emph{pair} (resp., \emph{sub-pair}) \emph{with linear systems} we mean the data
\[
(X, \mathcal{H}+ B)=(X, \sum a_i \mathcal{H}_i+\sum b_j B_j)
\]
where $X$ is
a normal projective variety, $\mathcal{H}_i$ is a movable (not necessarily complete) linear system, $B_j$ is a Weil divisor, $a_i, b_j\in \mathbb{Q}$ and $a_i, b_j\in [0, 1]$ (resp., $a_i\in [0, 1]$, $b_j\in (-\infty, 1]$), and % such that
%each $\mathcal{H}_i$ is either zero-dimensional, or movable,
%their general elements are $\QQ$-Cartier divisors, as well as
 for elements $H_i$ of $\mathcal{H}_i$ the divisor
\[
K_X + \sum a_i H_i+\sum b_j B_j
\]
is $\QQ$-Cartier (note that the latter condition does not depend on the choice of $H_i$).
We call the sum~$\sum b_j B_j$ \emph{the fixed divisor} of the pair.
If an algebraic group $G$ acts on the variety $X$, the sum of~linear systems $\sum a_i \mathcal{H}_i$ is $G$-invariant, and the sum of divisors $\sum b_i B_j$ is $G$-invariant, we call the pair (resp., sub-pair) $(X, \mathcal{H}+ B$) \emph{a $G$-pair with linear systems} (resp., \emph{a $G$-sub-pair with linear systems}).
\end{definitionsec}

By the \emph{log pull-back} of a $G$-sub-pair with linear systems
\[
(X, \mathcal{H}+ B)=(X, \sum a_i \mathcal{H}_i+\sum b_j B_j)
\]
via a projective birational $G$-morphism $f\colon Y\to X$
we mean the expression (which is not, generally speaking, a sub-pair)
\begin{equation}
\label{eq-log-pullback-linear-system}
(Y, \mathcal{H}_Y + B_Y) =
(Y, \sum a_i \widetilde{\mathcal{H}}_i + \sum b_j \widetilde{B_j}+ \sum c_k E_k)
\end{equation}
where $\mathcal{H}_Y=\sum a_i \widetilde{\mathcal{H}}_i $, $B_Y=\sum b_j \widetilde{B_j}+ \sum c_k E_k$,
and the formula
\begin{equation}
\label{eq-log-pullback-linear-systems}
K_Y + \sum a_i \widetilde{\mathcal{H}}_i + \sum b_j \widetilde{B_j}+ \sum c_k E_k = f^* (K_X + \sum a_i \mathcal{H}_i+\sum b_j B_j)
\end{equation}
holds, the divisors $E_k$ are $f$-exceptional, the linear systems $\widetilde{\mathcal{H}}_i$ are strict transforms of the linear systems $\mathcal{H}_i$, and $\widetilde{B}_j$ are strict transforms of $B_j$,
%$\mathcal{H}_Y=\sum a_i \widetilde{\mathcal{H}}_i$ and $B_Y= \sum b_j \widetilde{B_j}+ \sum c_k E_k$.

Let $(X, \mathcal{H}+ B)=(X, \sum a_i \mathcal{H}_i+\sum b_j B_j)$ be a $G$-sub-pair with linear systems. %where $\mathcal{H}_i$ are movable and $B_j$ are divisors,
%i.\,e. zero-dimensional linear systems.
Then there exists a \emph{$G$-equivariant} (cf. \cite{We55}) \emph{log resolution} of $(X, \sum a_i \mathcal{H}_i+\sum b_j B_j)$, that is, a smooth projective $G$-variety $Y$ together with a projective birational $G$-morphism $f\colon Y\to X$ such that
in~the formula~\eqref{eq-log-pullback-linear-systems}, the linear systems $\widetilde{\mathcal{H}}_i$ are base-point-free, and the divisor $\sum \widetilde{H_i} +\sum \widetilde{B_j} + \sum E_k$ has simple normal crossings for a general element $\widetilde{H_i}$ of the linear system $\widetilde{\mathcal{H}}_i$.

We define different types of singularities (lc, klt, plt, dlt) for sub-pairs with linear systems in~analogy with the case of pairs, see Section~\ref{subsec-types-of-singularities}. More precisely, for a log resolution $f\colon Y\to X$ of~the sub-pair $(X, \mathcal{H}+ B)$ defined above, we define \emph{a log discrepancy} of a prime divisor $D$ on $Y$ with respect
to $(X, \mathcal{H}+ B)$, denoted by $a(D, X, \mathcal{H}+ B)$,  as
\[
1 - \mathrm{coeff}_D (\sum b_j \widetilde{B_j}+ \sum c_k E_k)
\]
where $b_j, \widetilde{B_j}, c_k, E_k$ are as in \eqref{eq-log-pullback-linear-system}.

We
say that $(X, \mathcal{H}+ B)$ is lc (resp. klt) if $a(D, X, \mathcal{H}+ B)\geqslant 0$ (resp., $> 0$) and for any log resolution $f$ and for any divisor $D$.
We say that the pair $(X, \mathcal{H}+ B)$ is \emph{plt}, if~$a(D, X, \mathcal{H}+ B)>0$ holds for any log resolution $f$ and for any $f$-exceptional divisor $D$. %$\phi$.
We say that the pair $(X, \mathcal{H}+ B)$ is \emph{dlt}, if~$a(D, X, \mathcal{H}+ B)>0$ holds for some log resolution $f$ and for any $f$-exceptional divisor $D$.
% and for some log resolution. %~$\phi$.

We can extend in the natural way the notion of $G$-crepant birational equivalence to the case of $G$-sub-pairs with linear systems, see Definition \ref{def:crep-bir-isom}. We say that two $G$-sub-pairs with linear systems $(X, \mathcal{H}+ B)$ and $(X', \mathcal{H}'+ B')$ are \emph{$G$-crepant birationally equivalent} (or simply \emph{$G$-crepant equivalent}) if there exists the following diagram
\begin{equation}
\begin{tikzcd}
& (X'', \mathcal{H}''+ B'') \ar[rd, "\psi"] \ar[dl, swap, "\phi"] & \ \\
(X, \mathcal{H}+ B) \ar[rr, dashed, "\alpha"] & & (X', \mathcal{H}'+ B')
\end{tikzcd}
\end{equation}
where $\alpha$ is a birational $G$-map, $\phi$ and $\psi$ are birational $G$-contractions,
$(X'', \mathcal{H}''+ B'')$ is the log pullback of $(X, \mathcal{H}+ B)$ via the map $\phi$, and $(X'', \mathcal{H}''+ B'')$ is also the log pullback of $(X', \mathcal{H}'+ B')$ via the map $\psi$. %We can formulate this definition in the $G$-setting as well.

\begin{definitionsec}
We define the dual complex of an lc sub-pair with linear systems $(X, \mathcal{H}+ B)$ as~follows. Let $f\colon Y\to X$ be a log resolution of $(X, \mathcal{H}+ B)$, and let $(Y, \mathcal{H}_Y + B_Y)$ be its log pullback, which is a sub-pair by the assumption that $(X, \mathcal{H}+ B)$ is lc. Then the dual complex of~$(X, \mathcal{H}+ B)$ denoted by $\mathcal{D}(X, \mathcal{H}+ B)$ is the dual complex of the sub-pair $(Y, B_Y)$ in the sense of~Section~\ref{sec-dual-complex}.
\end{definitionsec}

In Proposition \ref{prop-dual-complex-correct-for-pairs} below we show that the dual complex is independent of the choice of a log resolution, up to PL-homeomorphism.

\begin{remarksec}
If in the lc sub-pair with linear systems $(X, \mathcal{H}+ B)$, all the linear systems $\mathcal{H}_i$ are base-point-free, then the dual complex of $(X, \mathcal{H}+ B)$ coincides with the dual complex of the lc~pair~$(X, B)$.
\end{remarksec}

On the other hand, if the linear systems in $\mathcal{H}$ have base-points, the dual complex of $(X, \mathcal{H}+ B)$ may differ from the dual complex of $(X, B)$.

\begin{examplesec}
\label{ex-linear-system1}
Consider the pair $(\mathbb{P}^2, \mathcal{H})$ where $\mathcal{H}$ is the linear system of cubics singular at some fixed point $P\in \mathbb{P}^2$. Then the dual complex of $(\mathbb{P}^2, \mathcal{H})$ is a point, while the dual complex of $(\mathbb{P}^2,0)$ is empty.
\end{examplesec}

%the dual complex the pair $(X, \sum a_i H_i+\sum b_j B_j)$ where $H_i\in \mathcal{H}_i$ is a general element,

\begin{propositionsec}[{cf. \cite{dFKX17},\cite[Lemma 2.34]{FS23}}]
\label{prop-dual-complex-correct-for-pairs}
The PL-homeomorphism class of the dual complex of~an~lc~sub-pair with linear systems does not depend on the choice of a log resolution.
\end{propositionsec}
\begin{proof}
Let $f\colon Y\to X$ and $g\colon Z\to X$ be two log resolutions of an lc sub-pair with linear systems~$(X, \mathcal{H}+B)$. Let
$(Y, \mathcal{H}_Y+B_Y)$ and $(Z, \mathcal{H}_Z+B_Z)$ be the corresponding log pullbacks. Then there exists a log resolution $W$ of $(X, \mathcal{H}+B)$ such that $W$ dominates both $Y$ and $Z$. Let $(W, \mathcal{H}_W+B_W)$ be the log pullback of $(X, \mathcal{H}+B)$ on $W$. Since the linear systems $\mathcal{H}_Y$ and $\mathcal{H}_Z$ are base-point-free, we~have that $(W, B_W)$ is the log pullback of both $(Y, B_Y)$ and $(Z, B_Z)$. Hence, by \cite{dFKX17} we~have the PL-homeomorphisms
\[
 \mathcal{D}(Y, B_Y) \cong \mathcal{D}(W, B_W) \cong \mathcal{D}(Z, B_Z),
\]
and thus the dual complex of $(X, \mathcal{H}+B)$ does not depend on the choice of a log resolution, up~to~PL-homeomorphism.
\end{proof}

\begin{remarksec}
We note that a pair with a linear system is a special case of a generalized pair (see, for example, \cite[Example 4.5(4)]{B21b}). As such, Proposition 5.5 can be viewed as a special case of \cite[Lemma 2.34]{FS23}.
\end{remarksec}

We say that an lc $G$-pair $(X, \mathcal{H}+ B)=(X, \sum a_i \mathcal{H}_i+\sum b_j B_j)$ with linear system is \emph{log Calabi--Yau},~if
\[
K_X + \sum a_i H_i+\sum b_j B_j \sim_\mathbb{Q} 0
\]
where $H_i$ is a general element in $\mathcal{H}_i$.

\begin{definitionsec}
\label{defin-regularity-linear-systems}
We define the \emph{$G$-coregularity with linear systems} %$\mathrm{coreg}^{\mathrm{ls}}_G(X, \mathcal{H}+ B)$
as follows (cf. Definition \ref{defin-regularity}). Let $X$ be a variety of dimension $n$ endowed with an action of an algebraic group $G$.
By \mbox{the~\emph{$G$-regu}}\-\emph{larity} \emph{with linear systems} $\mathrm{reg}^{\mathrm{ls}}_G(X, \mathcal{H}+B)$ of an lc $G$-pair $(X, \mathcal{H}+B)$ with linear systems we mean the number $\dim \mathcal{D}(X, \mathcal{H}+B)$.

For $l\geqslant 1$, we define the $l$-th \emph{$G$-regularity with linear systems} of $X$ by the formula
\begin{multline*}
\mathrm{reg}^{\mathrm{ls}}_{G, l}(X) = \max \{ \mathrm{reg}^{\mathrm{ls}}_G(X, \mathcal{H}+B)\ |\ (X, \mathcal{H}+B)=\\
=(X, \sum a_i\mathcal{H}_i+\sum b_jB)\ \text{is a log Calabi--Yau $G$-pair with linear systems and } la_i, lb_j\in \mathbb{Z})\}.
\end{multline*}
Then the $G$-\emph{regularity with linear systems} of $X$ is
\[
\mathrm{reg}^{\mathrm{ls}}_G(X) = \max_{l\geqslant 1} \{\mathrm{reg}_{G, l}(X)\}.
\]
Note that $\mathrm{reg}^{\mathrm{ls}}_G(X)\in \{-1, 0,\ldots, \dim X-1\}$ where by convention we say that the dimension of~the empty set is $-1$. The \emph{$G$-coregularity with linear systems} of an lc $G$-pair with linear systems~${(X, \mathcal{H}+B)}$ is defined as the number $n-1-\mathrm{reg}^{\mathrm{ls}}_G(\mathcal{H}+B)$. Also, we define:
\[
\mathrm{coreg}^{\mathrm{ls}}_{G, l}(X) = n-1-\mathrm{reg}^{\mathrm{ls}}_{G, l}(X).%,\quad \quad \quad \mathrm{coreg}_G(X) = n-1-\mathrm{reg}_G(X).
\]
%If $G$ is trivial, we obtain the notions of regularity with linear systems $\mathrm{reg}^{\mathrm{ls}}(X)$ and coregularity with linear systems $\mathrm{coreg}^{\mathrm{ls}}(X)$ of $X$.
We define \emph{regularity with linear systems} (resp. \emph{coregularity with linear systems}) of $X$ as
$$
{\mathrm{reg}^{\mathrm{ls}}(X)=\mathrm{reg}^{\mathrm{ls}}_G(X)}
$$
(resp. $\mathrm{coreg}^{\mathrm{ls}}(X)=\mathrm{coreg}^{\mathrm{ls}}_G(X)$) for the trivial group $G$.
\end{definitionsec}

Clearly, one has $\mathrm{reg}^{\mathrm{ls}}_{G, l}(X)\leqslant \mathrm{reg}^{\mathrm{ls}}_{G, kl} (X)$ for any $k,l\geqslant1$.

\iffalse
 coregularity the pair $(X, \sum a_i H_i+\sum b_j B_j)$ where $H_i\in \mathcal{H}_i$ is a general element, see Section \ref{sec-dual-complex}.

We define \emph{$G$-regularity with pairs} $\mathrm{coreg}_G^{\mathrm{ls}}$ of a klt Fano variety $X$ as
\[
\mathrm{coreg}_G^{\mathrm{ls}}(X) = \mathrm{min}( \coreg_G^{\mathrm{ls}}(X, \mathcal{H}+ B)) \ K_X + \sum a_i H_i + B\sim_{\mathbb{Q}} 0 )
\]
where $H_i$ are general elements of $\mathcal{H}_i$.
\fi

\begin{examplesec}
\label{ex-with-infinite-group}
Let $X=\mathbb{P}^2$, and let $G$ be a stabilizer of a point $P$ on $\mathbb{P}^2$ in the group $\mathrm{PGL}(3, \mathbb{K})$. Consider an lc $G$-pair with linear systems $(\mathbb{P}^2, \mathcal{H}_1+\mathcal{H}_2)$, where $\mathcal{H}_1$ is a linear system of lines on~$\mathbb{P}^2$ that pass through $P$, and $\mathcal{H}_2$ is a linear system of conics on $\mathbb{P}^2$ that pass through $P$. Then the $G$-coregularity with linear systems of $(\mathbb{P}^2, \mathcal{H}_1+\mathcal{H}_2)$ is equal to $1$. Note that there are no $G$-invariant elements in $\mathcal{H}_1$ and $\mathcal{H}_2$, so the usual notion of coregularity could not be applicable here.
%Clearly, there are no $G$-invariant elements in the linear systems $\mathcal{L}$ and $\mathcal{H}$.
%$|-nK_X|$ for $n\geqslant 1$.
\end{examplesec}

\begin{propositionsec}
\label{proposition: coregularity linear systems}
Let $X$ be a $G$-variety, where $G$ is a finite group. Then $\mathrm{reg}_G(X)=\mathrm{reg}^{\mathrm{ls}}_G(X)$.
\end{propositionsec}

\begin{proof}
Consider a log Calabi--Yau $G$-pair $(X,\sum b_i B_i)$ on which $G$-coregularity of $X$ is attained. %Then, interpreting $B_i$ as zero-dimensional linear systems,
Then we can consider  $(X,\sum b_i B_i)$ as a log Calabi--Yau $G$-pair $(X,\mathcal{H}+\sum b_i B_i)$ with linear systems where~${\mathcal{H}=0}$. %One can bound $G$-coregularity for linear systems for $X$ by $\coregG(X)$ from above.
Thus we obtain $\mathrm{reg}_G(X)\leqslant \mathrm{reg}_G^{\mathrm{ls}}(X)$.

Now let
\[
(X, \mathcal{H}+B)=(X, \sum a_i \mathcal{H}_i+\sum b_j B_j)
\]
be a log Calabi--Yau $G$-pair with linear systems on which the $G$-coregularity with linear systems of~$X$ is attained. We can compute the dimension of the dual complex of $(X, \mathcal{H}+B)$ on a log resolution~$Y$ of $(X, \mathcal{H}+B)$. Let
\[
(Y, \mathcal{H}_Y + B_Y) =
(Y, \sum a_i \widetilde{\mathcal{H}}_i + \sum b_j \widetilde{B_j}+ \sum c_k E_k)
\]
be the log pullback of $(X, \mathcal{H}+B)$ as in \eqref{eq-log-pullback-linear-system}. Consider general elements $H_i\in\mathcal{H}_i$,
and define
$$
H_i^G=\sum_{g\in G}\frac{gH_i}{|G|}, \quad \quad \quad H^G = \sum_i H_i^G.
$$
Note that all the of components $H^G$ have coefficients in $[0, 1]$. Then
\[
(Y,\sum a_i \widetilde{H}_i^G + \sum b_j \widetilde{B_j}+ \sum c_k E_k)
\]
is a log smooth log Calabi--Yau $G$-pair which
is a log pullback of
\[
(X,\sum a_i H_i^G + \sum b_j B_j),
\]
where $\widetilde{H}_i^G$ is the strict transform
of $H_i^G$.
By definition, the dual complex of the pair with linear systems $(X, \mathcal{H}+B)$ is homeomorphic to the dual complex of $(Y,  B_Y )= (Y,\sum b_j \widetilde{B_j}+ \sum c_k E_k)$.
However, note that
\[
\dim \mathcal{D}(Y,\sum b_j \widetilde{B_j}+ \sum c_k E_k)\leqslant \dim \mathcal{D}(Y,\sum a_i \widetilde{H}_i^G + \sum b_j \widetilde{B_j}+ \sum c_k E_k).
\]
%Note also that coefficients of $a_i \widetilde{H}_i^G$, that is, $\frac{a_i}{|G|}$, are strictly less then $1$.
This shows that $\mathrm{reg}^{\mathrm{ls}}_G(X)\leqslant \mathrm{reg}_G(X)$, which completes the proof.
\end{proof}

\section{Examples and questions}
\label{sec-examples-questions}
In this section we collect examples and questions. %consider del Pezzo surfaces of degree $5$ or lower. For these cases the automorphisms group of a surface is finite...... \footnote{to be continued}

In \cite{GHK15} it is proven that any log Calabi--Yau surface pair $(X, D)$ of coregularity $0$ admits a toric model, that is, $(X, D)$ is crepant equivalent to a toric pair.
The following example shows that this is not true in the $G$-equivariant setting.

\begin{examplesec}
\label{ex-no-G-toric-model}
Let $X$ be a smooth del Pezzo surface of degree $2$, and $G\subset \mathrm{Aut}(X)$ be a subgroup generated by the Geiser involution $\sigma$. Let $D=D_1+D_2$ be the union of two $(-1)$-curves interchanged by $\sigma$. Note that $D_1$ and $D_2$ intersect at two points which implies $\mathrm{coreg}_G(X)=0$. Then $D\sim -K_X$, so $(X, D)$ is a log Calabi--Yau $G$-pair. It is well known that the surface $X$ is birationally $G$-super-rigid. In particular, $X$ is not $G$-birational to any toric surface with an action of the group $G$. So~the pair $(X, D)$ is not $G$-crepant equivalent to a toric pair with an action of the group $G$.
\end{examplesec}

For more results on the existence of toric models for higher-dimensional coregularity $0$ pairs, see~\cite{LMV24} and references therein.

Note that $G$-coregularity of a variety $X$ depends on both $X$ and $G$. Indeed, if we fix $X$ and consider different $G$ acting on $X$, then it follows from Proposition \ref{proposition: dP9-int}. Now let us fix $G$ and vary $X$ instead. Let $G$ be a finite cyclic group, then $G$-coregularity can be equal to $0$ (for example, if $G$ is a subgroup of a torus on a toric variety~$X$, cf. Proposition \ref{prop-toric-varieities}), or it can be greater than $0$ as the following example shows.

\begin{examplesec}
Let $X$ be a smooth del Pezzo surface of degree $1$, such that all singular elements in $|-K_X|$ are cuspidal curves.
Proof of~\cite[Proposition 2.3]{ALP24} implies that $\coreg(X) = 1$. We can consider a group $G \cong \CG_2$ generated by the Bertini involution. Then obviously one has~${\coregG(X) = 1}$.
\end{examplesec}

The following question is motivated by Theorem \ref{main-thm-del-pezzo}.

\begin{questionsec}
\label{question-coreg0-in-Pn}
Let $G$ be a finite group that acts faithfully on $\mathbb{P}^n$. Is the condition that $G$ normalizes a torus $(\mathbb{K}^*)^n$ in $\mathrm{Aut}(\mathbb{P}^n)$ equivalent to
$\mathrm{coreg}_G(\mathbb{P}^n)=0$?
\end{questionsec}

For a cyclic group $G$, Question \ref{question-coreg0-in-Pn} is related to Question 1.6 in \cite{LoZ24}. One can formulate Question \ref{question-coreg0-in-Pn} for an arbitrary toric varieity instead of $\mathbb{P}^n$.

We also ask whether is it possible to relate the notion of weakly exceptional quotient singularities with respect to the group $G$ with the notion of $G$-coregularity.
For the relation between the notions of exceptional quotient singularities and $G$-coregularity, see Proposition \ref{propint-lct}.

\begin{questionsec}
\label{question-weakly-exc-coreg}
Let ${G}$ be the image of $\overline{G}$ in $\mathrm{PGL}_n(\mathbb{K})$ via the natural projection map
$$
{\mathrm{SL}_n (\mathbb{K})\to  \mathrm{PGL}_n (\mathbb{K})}.
$$
Assume that the singularity $0\in\mathbb{A}^n/\overline{G}$ is weakly exceptional but not exceptional.
Does this condition impose any restrictions on $\mathrm{coreg}_G(\mathbb{P}^{n-1})$?
\end{questionsec}

We note that in  \cite[Question 1.5]{LX24} the authors ask how to describe the set of exceptional divisors of plt blow-ups in the dual complex of a klt singularity. This is relevant to Question \ref{question-weakly-exc-coreg}.

\end{document}